\documentclass[11pt,leqno]{amsart}
\usepackage{amsmath,amsfonts,amsthm,amscd,amssymb,mathrsfs,graphicx,xcolor}
\usepackage{epstopdf}
\numberwithin{equation}{section}
\usepackage{fullpage}
\newcommand\s{\sigma}

\renewcommand\d{\partial}
\renewcommand\a{\alpha}
\renewcommand\b{\beta}
\renewcommand\o{\omega}
\def\g{\gamma}
\def\t{\tau}

\def\l{\lambda}
\def\eps{\varepsilon }
\def\e{\varepsilon}

\renewcommand\d{\partial}
\renewcommand\a{\alpha}

\renewcommand\b{\beta}
\renewcommand\o{\omega}

\def\g{\gamma}
\def\t{\tau}

\def\eps{\varepsilon}
\def\e{\varepsilon}

\def\l{\lambda}
\newcommand\br{\begin{remark}}
\newcommand\er{\end{remark}}
\newcommand\bp{\begin{pmatrix}}
\newcommand\ep{\end{pmatrix}}
\newcommand{\be}{\begin{equation}}
\newcommand{\ee}{\end{equation}}
\newcommand\ba{\begin{equation}\begin{aligned}}
\newcommand\ea{\end{aligned}\end{equation}}

\newcommand{\bap}{\begin{app}}
\newcommand{\eap}{\end{app}}
\newcommand{\begs}{\begin{exams}}
\newcommand{\eegs}{\end{exams}}
\newcommand{\beg}{\begin{example}}
\newcommand{\eeg}{\end{exaplem}}
\newcommand{\bpr}{\begin{proposition}}
\newcommand{\epr}{\end{proposition}}
\newcommand{\bt}{\begin{theorem}}
\newcommand{\et}{\end{theorem}}
\newcommand{\bc}{\begin{corollary}}
\newcommand{\ec}{\end{corollary}}
\newcommand{\bl}{\begin{lemma}}
\newcommand{\el}{\end{lemma}}
\newcommand{\bd}{\begin{definition}}
\newcommand{\ed}{\end{definition}}
\newcommand{\brs}{\begin{remarks}}
\newcommand{\ers}{\end{remarks}}

\newtheorem{hypothesis}{Hypothesis}

\newcommand{\RR}{{\mathbb R}}

\newcommand{\NN}{{\mathbb N}}
\newcommand{\ZZ}{{\mathbb Z}}
\newcommand{\TT}{{\mathbb T}}
\newcommand{\CC}{{\mathbb C}}

\newcommand{\sC}{\mathscr{C}}
\newcommand{\sN}{\mathscr{N}}
\newcommand{\sQ}{\mathscr{Q}}
\newcommand{\const}{\text{\rm constant}}

\newcommand{\tl}{\tilde{\l}}
\newtheorem{theorem}{Theorem}[section]
\newtheorem{proposition}[theorem]{Proposition}
\newtheorem{corollary}[theorem]{Corollary}
\newtheorem{lemma}[theorem]{Lemma}

\theoremstyle{remark}
\newtheorem{remark}[theorem]{Remark}
\theoremstyle{definition}
\newtheorem{definition}[theorem]{Definition}

\newtheorem{example}[theorem]{Example}

\newtheorem{obs}[theorem]{Observation}
\newcommand\cA{{\mathcal { A}}}

\newcommand\cC{{\mathcal  C}}

\newcommand\cR{{\mathcal  R}}

\newcommand\cL{{\mathcal  L}}
\newcommand\cN{{\mathcal  N}}

\newcommand\cP{{\mathcal  P}}
\newcommand\cQ{{\mathcal Q}}
\newcommand\cO{{\mathcal O}}

\newcommand\cM{{\mathcal M}}

\newcommand{\bbT}{{\mathbb{T}}}

\newcommand{\beq}{\begin{equation}}
\newcommand{\eeq}{\end{equation}}

\newcommand{\Hx}{\hat{x}}
\newcommand{\Ht}{\hat{t}}

\usepackage{color,soul}

\title{Convective Turing Bifurcation}

\author{Aric Wheeler}
\address{Indiana University, Bloomington, IN 47405}
\email{awheele@iu.edu }
\thanks{Research of A.W. was partially supported
under NSF grant no. DMS-1700279.}
\author{Kevin Zumbrun}
\address{Indiana University, Bloomington, IN 47405}
\email{kzumbrun@indiana.edu} 
\thanks{Research of K.Z. was partially supported
under NSF grants no. DMS-0300487 and DMS-0801745.}

\begin{document}

\begin{abstract}
Following the approach pioneered by Eckhaus, Mielke, Schneider, and others for reaction diffusion systems
\cite{E,M1,M2,M3,S1,S2,SZJV}, we 
justify rigorously by Lyapunov-Schmidt reduction the formal amplitude (complex Ginzburg Landau)
equations describing Turing-type bifurcations of general reaction diffusion convection systems,
showing that small spatially periodic traveling wave solutions of the PDE lie asymptotically close to
spatially periodic traveling waves of the amplitude equations, with asymptotically nearby speeds.
Notably, our analysis includes also higher-order, nonlocal, and even certain semilinear hyperbolic systems.
This is the first step in a larger program, laying the groundwork for spectral stability analysis
	\cite{WZ1}, and, ultimately, treatment of systems possessing conservation laws \cite{WZ2,WZ3}.
\end{abstract}

\maketitle
\section{Introduction}
In this paper, motivated by modern problems in biomechanical pattern formation, we revisit the
problem of Turing bifurcation, posed originally in the idealized context of reaction plus diffusion \cite{T},
in the more general context of PDE including mechanical or {\it convective} effects.
Namely, generalizing tools developed in \cite{M1,M2,S1,S2,SZJV,S,MC} for reaction diffusion systems,
we carry out a rigorous version of the formal ``weakly unstable approximation,'' or multiscale expansion 
of Eckhaus \cite{E}, derived originally in the hydrodynamical context of flow about an airfoil,
to obtain a complete description in terms of periodic traveling waves of
the associated ``amplitude equation'' \cite{vH,KSM,M3} consisting of the complex Ginzburg-Landau equation.
For the $O(2)$ symmetric reaction diffusion case of the references, 
this reduces to the real Ginzburg-Landau equation.

More precisely, we show that, near Turing bifurcation,
small spatially periodic traveling wave solutions of the PDE lie asymptotically close to
spatially periodic traveling waves of the associated complex Ginzburg-Landau equation, 
with asymptotically nearby speeds.
In a companion paper \cite{WZ1}, we show that spectral and time-asymptotic nonlinear stability
of bifurcating spatially periodic traveling waves is likewise predicted by the corresponding properties
of their complex Ginzburg-Landau approximants, completing the remaining part of the program of \cite{M1,M2,S1,S2}.

\bigskip

A question of substantial current interest is modeling of morphogenesis in both early and later stages:
e.g., vascularization, during branching, tubule formation, and remodeling/angiogenesis.
As described in \cite{MO,Ma,Mai,SBP,P}, the basic reaction diffusion model of Turing \cite{T} has
given way to various mechanochemical and hydrodynamical models of form 
\be\label{eq1}
\d_t w+\d_x f(w)=r(w)+ \d_x(b(w)\d_x w),
\ee
incorporating also convection, where $r$ and $b$ may in general be of full or partial rank.
For example, a simple version (neglecting shear forces) of the hydrodynamic vasculogenesis model of \cite{SBP} is
	\ba\label{hydro}
	\partial_t n + \nabla \cdot (nu)&=0,\quad
	\partial_t (nu) + \nabla \cdot (nu \otimes u - \bbT_n) 
	=n\nabla c -\tau_0 nu,\quad
	\partial_t c - \Delta c+ \tau_1 c=n,
\ea
where $n$ is density of endothelial cells (EC) lining the interior of blood vessels, 
assumed to be carried passively by the extracellular matrix (ECM)
of smooth muscle cells in the vessel wall \cite{WM}, $u$ is displacement of 
ECM, $\bbT_n$ is cell stress, and $c$ is concentration of chemical attractant.
	For the simplest choice $\bbT_n=\nabla \psi(u)$,
this has been reported to numerically reproduce structures resembling early in vitro networks;
cf.  \cite[Fig. 2, p. 11]{SBP}, \cite[Fig. 3, p. 551]{P}.

The earlier Murray-Oster model \cite{MO} for vasculogenesis is
\ba\label{MOeq}
	\partial_t n + \nabla \cdot (v_cn)&=0, \quad
\partial_t (m) + \nabla \cdot (m v_m)=0,\quad
\nabla \cdot (\bbT_n+ \bbT_m)+F&=0,
\ea
where $n$ and $m$ are density of EC and ECM, 
$v_m=\partial_t X_m$ is ECM velocity, $X$ denoting ECM displacement, $v_c$ is EC velocity determined through
physical/biological considerations as a function of other variables,
$\bbT_j$ are EC and ECM stresses, and $F$ is body force,
with the third equation representing total force balance.
This has been coupled in \cite{Ma} with a chemical attractant $c$ as in \eqref{hydro}(iii), with
reported encouraging correlation between numerical results and in vitro angiogenesis. 

A natural first step, but one that does not seem to have been addressed in the vasculogenesis literature, 
is to study ``initiation'' in the form of bifurcation from a constant solution, or ``Turing-type'' bifurcation,
via ``weakly unstable approximation,'' generalizing \cite{M1,M2,S1,S2,SZJV,S,MC},
of spatially periodic solutions, or ``patterns,'' of \eqref{eq1}.
This approach typically gives also stability information hence could be useful for
in vitro control/tissue engineering in helping choose parameters for which emerging network configurations are stable.

Of course, there are many other examples of pattern formation for models of form \eqref{eq1}, 
including shallow-water flow ($r$ of rank $n-1$) \cite{BJNRZ}; general conservation laws ($r$ of rank $0$) \cite{BJZ},
including general hydrodynamical flows; and flow in binary mixtures \cite{LBK,SZ}.
Hence, the study of Turing bifurcation for \eqref{eq1} is a problem of general interest 
independent of the context of biomorphology.

However, despite wide acceptance of Eckhaus' paradigm of weakly unstable dynamics governed approximately by
a complex Ginzburg-Landau equation \cite{E,AK,M3}, and numerous explicit computations carried out in both the $O(2)$
symmetric reaction diffusion case and the general $SO(2)$ case, and despite the development in \cite{M1,M2,S1,S2} 
of a general method based on Lyapunov-Schmidt reduction capable to rigorously justify
their implications for Turing bifurcation in terms of shape and time-asymptotic stability of bifurcating waves,
the implementation of this rigorous justification seems to have lagged behind.
Indeed, even in the $O(2)$ reaction diffusion case, for which the weakly unstable expansion reduces
to the real Ginzburg-Landau equation (rGL), rigorous justification of the expansion, 
in the global-in-time sense \cite{M1,M2,S1,S2} 
has so far been carried out completely only for a few specific models \cite{M2,S1,SZJV,S},
and none at all to our knowledge in the general $SO(2)$ case.
Thus, there appears to be a need for further analysis, even in the classical
(full-rank) case without conservation laws, most particularly in the presence of convection.
We address this here and in \cite{WZ1}, both for its individual interest in completing
the program of \cite{M1,M2,S1,S2},
and as preparation for the analysis in \cite{WZ2,WZ3} of systems possessing conservation laws.

\br\label{justrmk}
As discussed in \cite[\S 6]{M3}, there are a number of different senses in which
one might pursue rigorous verification of the complex Ginzburg Landau equation, of interest
in different settings.
These can be divided roughly into {\it finite-time approximation} properties for general solutions, 
and {\it global-in-time existence and behavior} for special solutions: the
former stating for all complex Ginzburg-Landau solutions in an appropriate space that there are nearby exact
solutions of the underlying PDE remaining close up to a given finite time $T$, corresponding
to $T/\eps^2$ in the Ginzburg-Landau scaling, where $\eps$ is the order of the bifurcation parameter;
and the latter stating for traveling-wave or periodic solutions of the complex Ginzburg-Landau
equation (cGL) that there exist nearby exact solutions in the same category, whose time-asymptotic stability
properties with respect to the underlying PDE moreover agree with those of the approximating Ginzburg-Landau
solution with respect to (cGL).
These may be recognized as different qualities of center manifolds in finite-dimensional ODE,
supporting the viewpoint \cite{M3} of Ginzburg-Landau as infinite-dimensional center manifold.
Here, \emph{we exclusively discuss the latter, global-in-time notion} relevant to Turing bifurcation.
As regards the former, finite-time approximation notion \cite[\S 6.2]{M3}, there exist a variety of works 
dating back to \cite{vH}, in rather complete generality. 
\er

\bigskip

In the present work, we begin a larger program on initiation in convective morphogenesis--
more generally, bifurcation from constant solutions for systems \eqref{eq1}--
with the analog of Turing's original problem for general PDE depending on a bifurcation parameter $\mu$,
proving existence and closeness to complex Ginzburg-Landau approximations of small
periodic traveling waves, for $\mu$ sufficiently near a bifurcation point $\mu=0$.
In the companion paper \cite{WZ1}, we show that spectral and nonlinear stability
are likewise well-predicted by the complex Ginzburg-Landau approximation, rigorously validating
the famous sideband stability criteria of  Eckhaus \cite{E} for general reaction convection diffusion
systems for which the reaction term is {\it full rank}.
In \cite{WZ2,WZ3} we extend our analysis to the case of {\it non-full-rank} systems with conservation laws, 
as in \cite{MO,Ma,Mai,SBP,P}.

Namely, assuming existence of a smooth family of constant solutions $w_\mu$,  and introducing $u:=w-w_\mu$,
we consider the family of perturbation equations in standard form 
\be\label{std}
u_t=L(\mu)u+\cN(u,\mu),
\ee
where $L(\mu)=\sum_{j=0}^m \cL_j(\mu)\d_x^j$ is a constant-coefficient differential operator
and $\cN$ is a general nonlinear functional of quadratic order in $u$ and $x$-derivatives,
under generalized {\it Turing assumptions} on the spectra of $L$ near the bifurcation point $\mu=0$,
or, equivalently, on the eigenvalues $\tilde \lambda_j(k,\mu)$ of the associated Fourier symbol
$S(k,\mu)=\sum_{j=0}^{m}\cL_j(\mu)(ik)^j$.

These assumptions, detailed in Hypothesis \ref{hyp:Turing} below, ensure that
i) except for a single pure imaginary eigenvalue $\tilde \lambda$ at $k=\pm k_*\neq 0$,
all eigenvalues of $S(k, \mu)$ have strictly negative real part; 
(ii) the symbol is strictly stable as $|k|\to 0$ or $\infty$, so that $0<|k_*|<\infty$; 
and (iii) $\Re \partial_\mu \tilde \lambda(k_*,0)>0$, so that there is a change in stability as $\mu$ crosses
zero from left to right. 
In particular, they imply that $\cL_0(0)$-- $r(w_0)$ in the case of \eqref{eq1}--
must be strictly stable, hence {\it full rank}.
As already noted, {\it this} (full rank) {\it condition is violated for the vasculogenesis models mentioned above},
hence the present study is a preliminary step toward the study of that more degenerate case
(see {\it Discussion}, below).

\subsection{Ginzburg-Landau approximation}\label{s:GLapprox}
Let $r$ denote the eigenvector of $S(k_*,0)$ associated with the critical eigenvalue $\tl(k_*,0)$,
so that (by complex conjugate symmetry, noting that $L$ is real-valued), $\tl(-k_*,0)=\overline{ \tl(k_*,0)}$, with
associated eigenvector $\bar r$.
Then, 
\be\label{exact}
u(x,t)=e^{i(k_*x + \Im \tl(k_*,0)t)}r+ c.c.
\ee
is an exact nondecaying spatially-periodic solution of the linearized equations 
$u_t=L(0)u$ at the bifurcation point $\mu=0$, where, here and elsewhere, $c.c.$ denotes complex conjugate.
By our spectral hypotheses, meanwhile, all other eigenmodes are time-exponentially decaying at varying rates.

With these preliminaries, the ``weakly unstable'' or ``weakly nonlinear'' expansion of Eckhaus \cite{E}
consists in seeking for $\mu=\eps^2 \ll 1$, formal asymptotic solutions of form
\be\label{form}
U^{\e}(x,t)=\frac{1}{2}\e A(\Hx,\Ht)e^{i\xi}r + \cO(\e^2) +c.c.,
\quad \xi=k_* \Big(x + \frac{\Im \tl(k_*,0)}{k_*}t \Big)
\ee
of the full nonlinear equation \eqref{std}, based on modulations with vaying amplitude $A$ of
the neutral linear solution \eqref{exact} at $\mu=0$, with $(\Hx,\Ht)$ an appropriate
rescaled moving coordinate frame.
As described, e.g., in \cite{M3}, the equations close under the choice of coordinates
\be\label{scale}
\Hx=\e(x+ \Im\d_k\tl(k_*,0)t), \quad  \Ht=\e^2t,
\ee
yielding an {\it amplitude equation} consisting of the complex Ginzburg-Landau equation (cGL):
\begin{equation}\label{eq:cGL}
	A_{\Ht}=-\frac{1}{2}\d_k^2\tl(k_*,0)A_{\Hx\Hx}+\d_\mu\tl(k_*,0)A+\gamma|A|^2A,
\end{equation}
where the Landau constant $\gamma\in \CC$ is determined by the form of the nonlinearity $\cN$
together with linear information about the spectral structure of $S(k_*,0)$, see \eqref{eq:generalgamma} in Lemma \ref{lem:GeneralGamma} and \cite[\S 2.2]{M3} for the formula for $\gamma$.
%TODO: maybe update to more general discussion when it is added?-K
The different speeds $-\frac{\Im \tl(k_*,0)}{k_*}$ vs. $-\Im\d_k\tl(k_*,0)$ in
the moving frames $\xi$ and $\Hx$ correspond to {\it phase} vs. {\it group} velocities of
the underlying linear exponential solutions $e^{ik_* x+ \tl(k_*,\mu)t}$.

In the $O(2)$ symmetric reaction-diffusion case, invariant under reflection $x\to -x$ as well as translation-- 
more generally, when both $L$ and $\cN$ depend only on even order derivatives of $w$,
$\tl(k_*,0)=0$ and \eqref{eq:cGL} reduces to the real Ginzburg-Landau equation (rGL):
\be\label{eq:rGL}
A_{\Ht}= c_1 A_{\Hx \Hx} + c_2 A + \gamma |A|^2 A; \qquad c_j, \gamma \in \RR
\ee
and $\xi$, $\Hx$ to the stationary frames $k_*x$, $\eps x$.
The first may be seen by the fact that $S(k_*,0)$ is real, so that the assumption of a single
imaginary eigenvalue $\tl(k_*,0)$ implies $\tl(k_*,0)=0$;\footnote{
For $\tl(k_*,0)\neq 0$, there is a higher codimension bifurcation 
involving counterpropagating waves \cite{CK,PYZ,AK}.}
the second by the fact that reflection invariance is inherited in \eqref{eq:cGL} as
invariance under complex conjugation, and $\xi$, $\Hx$ to the stationary frames $k_*x$, $\eps x$.
See, e.g., \cite{M1,M2,SZJV}, for further discussion.

\bigskip
As described in surveys \cite{AK,vSH,M3}, the complex Ginsburg-Landau equation \eqref{eq:cGL} supports
a rich variety of coherent structures, including front, pulse, and periodic,
as well as source/sink type solutions \cite{SSc,DSSS,BNSZ}.
For our purposes, the relevant ones ones are periodic solutions 
\be\label{GLper}
A=e^{i(\kappa \Hx+ \omega \Ht)} \alpha,  \quad \alpha\equiv \const,
\ee
corresponding 
through \eqref{form}--\eqref{eq:cGL} to
approximate time- and spatially-periodic traveling waves
$$
\begin{aligned}
	 U^{\e}(x,t)&=\frac{1}{2}\e \alpha  e^{i(kx + \Omega t)}r + \cO(\e^2) +c.c.\\
\end{aligned}
$$
with spatial and temporal wave numbers $k=k_*+\eps \kappa$ and
\be\label{st}
\Omega= \Im \tl(k_*,0) + \eps\kappa \partial_k\tl(k_*,0) + \eps^2 \omega . 
\ee

These may be seen to be stationary in the original rest frame $\xi$ for all $\eps$ if and only if 
$\omega=0$ and $\frac{\Im \tl(k_*,0)}{k_*}=\partial_k \Im \tl(k_*,0)$:
that is, the Ginzburg-Landau solution \eqref{GLper} is stationary {\it and} group and phase velocities coincide.
Otherwise, they are not all stationary in {\it any} one frame as $\eps$ is varied.

Plugging \eqref{GLper} into \eqref{eq:cGL} for $\alpha\neq 0$ gives the {\it nonlinear dispersion relation}
\be\label{nld}
i\omega= -\frac12 \d_k^2\tl(k_*,0) \kappa^2 +\d_\mu\tl(k_*,0)+\gamma|\alpha|^2,
\ee
characterizing $\alpha$ and $\omega$ as functions  of $\kappa$
\ba\label{alphaomega}
|\alpha|&= \sqrt{ \Re \gamma^{-1}\Big( \frac12 \d_k^2\Re \tl(k_*,0) \kappa^2 -\Re \d_\mu\tl(k_*,0) \Big)},\\
\omega&= -\frac12 \Im \d_k^2\tl(k_*,0) \kappa^2 +\Im \d_\mu\tl(k_*,0)+\Im \gamma|\alpha|^2,
\ea
from which we see that solutions exist under the {\it supercriticality condition}
$\Re \gamma \Re \d_\mu\tl(k_*,0)<0$, within range
\be\label{krange}
\kappa^2 < 2 \Re \d_\mu\tl(k_*,0)/ \d_k^2\Re \tl(k_*,0),
\ee
and are stationary if and only if $\tl(k_*,0)$, $\d_\mu\tl(k_*,0)$, and $\gamma$ are common complex multiples of reals.

\subsection{Main results}\label{s:main} 
With these preparations, our main results are as follows.

\begin{theorem}[Expansion \cite{SS,NW,M3}]\label{maincGL}
Under Turing Hypotheses \ref{hyp:Turing},
for quasilinear nonlinearity $\cN$ and $\mu=\eps^2$, for any smooth solution of \eqref{eq:cGL}
that is uniformly bounded in $C^s$, $s$ sufficiently large, for $0\leq \Ht\leq T$, 
or equivalently $0\leq t\leq T/\eps^2$, expansion \eqref{form}-\eqref{scale}, augmented by
an appropriately chosen smooth $\eps^2$ order corrector, is uniformly 
valid to order $\eps^3$, i.e., its truncation error as an approximate solution of \eqref{std}
is bounded by a constant multiple of $\eps^4$, for $0\leq \hat t \leq T$.
\end{theorem}

%TODO: cite M3? or? NO, below comments in rmk are better.
\begin{theorem}[Existence]\label{mainLS}
Under Turing Hypotheses \ref{hyp:Turing},
for quasilinear nonlinearity $\cN$ and $\mu=\eps^2$, for any $\nu_0>0$ there exists 
$\eps_0$ such that for $\eps\in[0,\eps_0)$ 
and $\kappa^2 \leq (1-\nu_0) 2 \Re \d_\mu\tl(k_*,0)/ \d_k^2\Re \tl(k_*,0)$ there exists
	a unique (up to translation, i.e., up to choice of $\alpha$) 
	small spatially periodic traveling-wave solution $\bar U^\eps (kx+\bar \Omega t)\not \equiv 0$ 
of \eqref{std}, $\bar U$ $2\pi$-periodic, with $k=k_*+ \eps \kappa$, satisfying
\ba\label{soln}
	\bar U^{\e}(z)&= \Big(\frac{1}{2}\e \alpha  e^{iz}r + c.c.\Big) +\cO(\e^2),\\
	\bar \Omega &= \Big( \Im \tl(k_*,0) + \eps\kappa \partial_k\tl(k_*,0) + \eps^2 \omega \Big) + O(\eps^3),
\ea
where $\alpha\in \CC$ and $\omega\in \RR$ satisfy \eqref{alphaomega},
while for $\eps\in[0,\eps_0)$ and
$\kappa^2 \geq (1+\nu_0) 2 \Re \d_\mu\tl(k_*,0)/ \d_k^2\Re \tl(k_*,0)$ there exist no such small nontrivial solutions.
In the ($O(2)$-symmetric) generalized reaction diffusion case that $L$ and $N$ depend only 
on even derivatives or even powers of odd derivatives of $u$, $\bar \Omega\equiv 0$ and $\bar U^\eps$ is even for $\alpha\in \RR$.
\end{theorem}

These results are of two rather different types, the first concerning formal accuracy, or truncation error,
of the complex Ginzburg-Landau approximation for general solutions of (cGL) on a finite time-interval, 
and the second existence of and rigorous convergence error from nearby exact solutions of \eqref{std} 
for the special case of space-time periodic solutions \eqref{GLper}-\eqref{st} of (cGL).
In both cases, we show that our results remain valid under reasonable assumptions, also in the {\it nonlocal}
case that $L(k)$ a general Fourier multiplier; see Section \ref{sec:MGN}.
This may be useful in applications such as chemotaxis, water waves, etc.;
see, for example, \cite{BBTW,L} and references therein.

Theorem \ref{maincGL} was established by classical matched asymptotic analysis
in \cite{SS} and \cite{NW} for plane Poiseuille flow and 
Rayleigh-Benard convection; the general case is treated in \cite[\S 2.2]{M3}.
For completeness, and as preparation for the analysis in Theorem \ref{mainLS} and
companion paper \cite{WZ1}, we reprove the theorem here step by step, in full detail.
The treatment of nonlocal equations in Section \ref{sec:MGN} and expansion to all orders in Section \ref{s:allorders} 
may likewise be of interest; see also the discussion of nonresonant semilinear hyperbolic problems in 
Remark \ref{wrinklermk}.
Theorem \ref{maincGL} is established in Theorem \ref{3.1} in the simple case of a nonlinearity
that is a function of $u$ alone, and extended to all orders
and general quasilinear nonlinearities in Theorem \ref{allordersthm}.

Theorem \ref{mainLS} so far as we know is new in the general $SO(2)$ (translation- but not reflection-invariant) 
case- at least in its full details- and certainly in its method of proof.
The latter, similarly as in \cite{M}, uses Lyapunov-Schmidt reduction
to a codimension two $SO(2)$ bifurcation in two dimensions
parametrized by $(\mu, \delta)$, where $\delta$ is a free parameter
allowing for variation in speed, tracing through this process and the matched asymptotic steps
of Theorem \ref{maincGL} to verify that the resulting reduced system matches to lowest order
the rotating-wave system \eqref{nld} for the complex Ginzburg-Landau equation \eqref{eq:cGL}.
It is established for nonlinearies that are functions of $u$ alone in 
Theorem \ref{thm:LSReduction} and Corollary \ref{mainprop}, and extended to 
general quasilinear nonlinearities in Section \ref{s:LS2}. 
It is extended to nonlocal nonlinearities in Remark \ref{genrmk}.
We note the interesting subtlety that, without $O(2)$ symmetry, 
{one cannot conclude existence of stationary solutions},
even in the case that the approximating (cGL) solution is stationary,
but only traveling waves with slow, $O(\eps^3)$ speed.

\br\label{existrmk}
An alternative approach to the proof of Theorem \ref{mainLS}, as described in
\cite[Case 2, \S 6.1]{M3} is Kirchg\"assner reduction, or ``spatial dynamics'' \cite{Ki}, in 
which one seeks time-periodic solutions by reduction to a center manifold ODE in
$x$ within the space of time-periodic functions.
However, this approach, though elegant, does not seem to yield stability information.
Ultimately, both approaches rely on reduction to a two-dimensional ODE with $SO(2)$ invariance inherited
from translational invariance in the original problem: temporal in the spatial dynamics setting 
and spatial in our setting of classical Lyapunov-Schmidt reduction, with the 
main technical tasks being, first, to confirm that the resulting reduced systems match to up to
a small error the analogous rotating-wave system for \eqref{eq:cGL} and, second, to show by unfolding of the 
bifurcation that this small error in the models indeed translates to a small error in the solutions.
%NOTE: neither done in \cite[\S 2.2]{M3}... only the general concept is described. This is why ``difficulties
%are removed...'' I believe, I expect they are still there in the actual computations, even for spatial 
%dynamics, just weren't carried out.
\er

\subsection{Discussion and open problems}\label{s:disc}
Theorem \ref{mainLS} yields rigorous global-in-time accuracy of special solutions of \eqref{eq:cGL} as
approximate solutions of \eqref{std}.
The complementary question of validity for bounded time of general solutions of \eqref{eq:cGL}
has been studied for real and complex Ginzburg-Landau in, e.g., \cite{CE,S3,KSM} and \cite{vH,M3}, 
for various classes of initial data on \eqref{eq:cGL},\footnote{
	See also \cite{KT} for justification of (rGL) in the nonlocal case, for a model Swift-Hohenberg type equation. We note that the analysis \cite{KSM} in the case of cubic order nonlinearity does not require smoothing, applying also in the hyperbolic case. The analysis of \cite{vH} is restricted to the case of
(exactly) quadratic nonlinearity.}
with the typical result that there exists an
exact solution of \eqref{std} remaining $\cO(\eps^2)$ close to the corresponding
$\cO(\eps)$ term in \eqref{form} on a bounded time interval $\Ht\in [0,T]$, or, in original coordinates
$ t\in [0, T/\eps^2]$.

We note for the special solutions of Theorem \ref{mainLS} the convergence error $|U^\eps -\bar U^\eps|$
is $\cO(\eps^2)$ for 
$$
|\partial_x \bar U| |\Omega -\bar \Omega|t\lesssim \eps^2,
$$
or, using $\partial_x \bar U\sim \eps$, $|\Omega -\bar \Omega|t\lesssim \eps$, 
for $t\lesssim \eps^{-2}$, in agreement with the general result of \cite{vH}.

The variation $\bar \Omega\neq \Omega$ in speed between exact and approximate solutions is the main technical
difference between the $SO(2)$ invariant case treated here and the $O(2)$ invariant case treated in previous
works. A related result in the fixed-period case $\kappa\equiv 0$, 
is the treatment of transverse $SO(2)$ Hopf bifurcation in \cite{M,BMZ} of magnetohydrodynamic shock waves in a channel.
This difference may be understood (cf. \cite{M})
by comparing $O(2)$ vs. $SO(2)$-invariant ODE in the plane, or, writing
in complex form: $\dot A= f(|A|)A$ with $f$ real- vs. complex-valued.
In the first place, one may seek steady solutions $A\equiv \alpha$ with $\alpha \equiv \const$ 
by solving the scalar equation $f(|A|)=0$;
in the second, one seeks rotating solutions $A=e^{\omega t}\alpha$ by solving
the scalar equations $\Re f=0$ and $\Im f=\omega$, resulting in general in nonstationary solutions $\omega\neq 0$.
The speed $\omega$ serves as an additional bifurcation parameter along with $\mu$ in the $SO(2)$ case, 
making this a codimension-two bifurcation as compared to the codimension-one bifurcation of the $O(2)$ case.
This type of computation may be found, repeated, throughout our analysis
of both formal expansion and Lyapunov-Schmidt reduction, in solving the $2$-dimensional $SO(2)$-invariant
equations to which both ultimately reduce.

\bigskip

As regards further directions for study, we mention, first, the physically important
question of {\it time-asymptotic stability.}
Stability of periodic solutions \eqref{GLper} as solutions of (cGL) can be explicitly determined \cite{AK,TB},
leading to the formal ``Eckhaus criterion'' for 1-D stability of exact solutions 
\eqref{soln} as solutions of \eqref{std}.
Indeed, this could be partly validated in principle using the existence theory developed here
via the Whitham modulation criterion for the bifurcating waves \cite{W,JNRZ,SSSU}, a 
low-frequency {\it necessary condition}
for stability depending only on existence theory and spectral information of the neutral, ``translational''
eigenmodes $\partial_x \bar U^\eps$.
We shall not pursue that, but instead carry out a full (necessary and sufficient) stability 
analyis in \cite{WZ1} generalizing to the complex Ginzburg-Landau case the results of
\cite{S1,SZJV} for the real Ginzburg-Landau case.

Another interesting direction for further exploration would be rigorous validation, either for exact periodic
solutions, or general solutions of (cGL) on time interval $[0, T/\eps^2]$,
of higher-order expansions of \eqref{soln} as constructed in Section \ref{s:allorders}.

For the applications to vasculogenesis models that we have in mind, it is important also to extend
to the case that $r$ in \eqref{eq1} have incomplete rank, in particular that 
$$
w=\begin{pmatrix}w_1 \\  w_2 \end{pmatrix}, \quad 
r=\begin{pmatrix}0 \\ r_2 \end{pmatrix}, \quad 
B=\begin{pmatrix}B_{11} & B_{12} \\ B_{21}&B_{22} \end{pmatrix},
$$
with $B_{11}$ full rank.
Interestingly, the existence problem for this case may be treated by the theory already developed here.
For, integrating the $w_1$ equation gives a family of conservation laws
$f_1(w)-B_{11} \partial_x w_1 - B_{12}\partial_x w_2\equiv  p$, for $p$ a vector of parameters of dimension
$\dim w_1$.
Solving these relations using Fourier inversion combined with the implicit function theorem,
we may obtain $w_1$ as a nonlocal function of $w_2$, yielding a family of nonlocal problems of the type
treated in Section \ref{sec:MGN} smoothly parametrized by $p$, and satisfying Turing Hypotheses \eqref{hyp:Turing}.
Applying the theory already developed, we find that small periodic traveling waves are given by
$p\equiv \const$ and $w_2$ a solution of form \eqref{soln}.

This generalizes existence results obtained in \cite{MC,S} in the case of a single conservation law for a
model $O(2)$-invariant Swift-Hohenberg type equation.
Continuing this analogy, we derive in \cite{WZ2,WZ3} also a description of behavior/stability analogous to that 
of \cite{MC,S} in terms of amplitude equations coupling (cGL) and conserved quantities,
generalizing results of \cite{HSZ} for B\'enard-Marangoni and thin-film flow.
As noted in \cite{HSZ}, under the influence of convection, these amplitude equations in general become
{\it singular}, exhibiting $\eps^{-1}$ order convective mixing in ``mean modes'' associated with conservation laws,
a circumstance that greatly complicates the analyis of stability and behavior.
It is this novel aspect, and the associated lack of Eckhaus-type stability analysis, we believe,
that has up to now prevented the application of weakly unstable approximation techniques
to the problem of initiation in vasculogenesis.
An important further extension would be to treat the case of incomplete parabolicity $\det B=0$ occurring for
actual physical models.

Finally, and more speculatively,
an important challenge is to go beyond the initialization phase to describe
longer-term/larger scale development of vascular structure: that is, the slower time-scale ``emergent structure''
not directly programmed by the model/cell genetics.
There are many possible dynamical systems mechanisms by which such multiscale dynamics can occur;
see, e.g., \cite{CMM,BW1,BW2}.
As a first step, we have in mind to apply modulation techniques like those developed in \cite{W,DSSS,JNRZ,SSSU,MZ}
for the description of behavior of ``fully-developed'' large-amplitude patterns.
As a model in one dimension, see for example the ``coarse-grained'' description of 
behavior of periodic Kuramoto--Sivashinsky cells
in \cite{FST}, and its wide generalizations in \cite{JNRZ}.

%TODO: add later if necessary (e.g. thanking referees...-K
%\medskip
%{\bf Acknowledgement.}  We thank... etc.
%Maybe Guido if he has comments?

\section{Preliminaries}\label{sec:Prelim}
%TODO: $S(k,\mu)$ and $L(k, \mu)$ are quite different, confusing... is this really what we want?-K) OK I GUESS.
To begin, we consider the following system in a neighborhood of a Turing bifurcation:
\begin{equation}\label{eq:MasterEqn}
	u_t=L(\mu)u+\cN(u),
\end{equation}
where $\cN:\RR^n\rightarrow\RR^n$ is a smooth nonlinear function of quadratic order in $u$,
and $L(\mu)$ is a constant-coefficient linear operator
\begin{equation}\label{eq:LinOpDef}
	L(\mu)=\sum_{j=0}^m \cL_j(\mu)\d_x^j,
\end{equation}
where $\cL_j(\mu)$ is a $C^{\infty}$ function of $\mu$ with values in $M_n(\RR)$, the set of $n\times n$ real matrices.
We define the associated Fourier symbol
\be\label{sym}
	S(k,\mu)=\sum_{j=0}^{m} \cL_j(\mu)(ik)^j.
\ee

\begin{remark}
	A typical source of these types of systems are reaction-diffusion or reaction-diffusion-convection systems. It is a straightforward generalization to allow $\cN$ to depend on $\mu$, the only change being some added bookkeeping.
	We shall see later that $\mu$-dependence in $\cN$ enters the analysis at higher order, affecting neither amplitude equations nor nonlinear existence.
	%TODO: make a rmk on this to cite here...-K.  NO, NO NEED.
\end{remark}

To find periodic solutions to \eqref{eq:MasterEqn}, we rescale $\xi=kx$ and define the modified linear operator
\begin{equation}\label{eq:KLinOpDef}
	L(k,\mu)=\sum_{j=0}^{m}k^j\cL_j(\mu)\d_\xi^j
\end{equation}
The advantage of this change of coordinates is that all periodic solutions are now supported on the same fixed integer lattice in Fourier space.

The following conditions codify our notion of generalized Turing bifurcation.

\begin{hypothesis}\label{hyp:Turing}
The symbol $S(k,\mu)$ and its eigenvalues $\{\tl(k,\mu),\tl_2(k,\mu),...,\tl_n(k,\mu) \}$ satisfy:\\

\noindent
	(H1) For $\mu<0$ and all $k\in\RR$, $\sigma(S(k,\mu))\subset\{z\in\CC:\Re z<0 \}$.\\
	(H2) For $\mu=0$ there is a unique $k_*>0$ such that $\Re\tl(k_*,0)=0$ and for $2\leq j\leq n$ $\Re\tl_j(k_*,0)<0$.\\
	(H3) For $\mu=0$ and all $k\not=\pm k_*$, we have that $\Re\tl(k,0)<0$ and for $2\leq j\leq n$ $\Re\tl_j(k,0)<0$.\\
	(H4) $\Re\d_\mu\tl(k_*,0)>0$, $\Re\d_k\tl(k_*,0)=0$ and $\Re\d_k^2\tl(k_*,0)<0$.
\end{hypothesis}

%\begin{remark}\label{placermk}
	For general results that hold for $|k-k_*|\gg1$, we will for simplicity denote $\tl(k,\mu)=\tl_1(k,\mu)$,
	as in this regime $\tl(k,\mu)$ behaves similarly as other $\tl_j(k,\e)$.

There are some simple conditions to on the symbol of a differential operator that ensure (H3), at least for $|k|\gg1$ and for $|k|\ll1$, which we describe in the following proposition.
\begin{proposition}\label{prop:NecSuf}
	We have the following criteria for satisfaction of (H1)--(H3).
	\begin{enumerate}
		\item For $|k|\ll 1$, (H1) and (H3) are equivalent to $\sigma(\cL_0(\mu))\subset\{z\in\CC:\Re z<0\}$.
		\item If $m$ is even, then $\sigma\left((-1)^{\frac{m}{2}}\cL_m(\mu)\right)\subset\{z:\in\CC:\Re z<0 \}$ is sufficient for (H1) and (H3) in the regime $|k|\gg1$.
		\item If $m$ is odd, then necessarily $\sigma(\cL_m(\mu))\subset\RR$; moreover if we in addition assume that $\cL_m(\mu)$ is diagonalizable, and $\ell^{\infty}_j$ and $r^{\infty}_j$ denote the left and right eigenvectors of $\cL_m(\mu)$, then $\ell^\infty_j(-1)^{\frac{m-1}{2}}\cL_{m-1}(\mu)r^\infty_j<0$ implies (H1) and (H3) in the regime $|k|\gg1$.
	\end{enumerate}
\end{proposition}

\begin{remark}
	Note that the second condition is essentially equivalent to saying that $L(\mu)$ is an elliptic operator when $m$ is even. The first assertion disallows conserved quantities, c.f. \cite{MC,S}.
\end{remark}

\begin{proof}
	For (1), this follows immediately from the fact that $\{\tl(k,\mu),\tl_2(k,\mu),...,\tl_n(k,\mu) \}\rightarrow \sigma(\cL_0(\mu))$ as $k\rightarrow 0$ by the continuity of the spectrum.

In order to attack (2) and (3), we rescale $S(k,\mu)$ as
	\begin{equation}\label{eq:TildeS}
		S(k,\mu)=k^m\left(i^m\cL_m(\mu)+\frac{i^{m-1}}{k}\cL_{m-1}(\mu)+...+\frac{1}{k^m}\cL_0(\mu) \right)=k^n\tilde{S}(\frac{1}{k},\mu)
	\end{equation}
	Now, for $\eta:=\frac{1}{k}$ we have $\sigma(\tilde{S}(\eta,\mu))\rightarrow\sigma(i^m\cL_m(\mu))$ as $\eta\rightarrow0$. First assume that $m$ is even, then we have $i^m=(-1)^{\frac{m}{2}}$ and so the claim in (2) follows by continuity of the eigenvalues and the observation that $k^m\geq0$ for all $k\in\RR$.

	To complete the argument, we now assume that $m$ is odd. Let $\tl^\infty_j(\mu)=\a_j+i\b_j$, $j=1,...,n$, be the eigenvalues of $\cL_m(\mu)$ and $\ell^\infty_j$ and $r^\infty_j$ be the associated left and right eigenvectors. By the limiting argument for $\tilde{S}(\eta,\mu)$, we have the asymptotic expansion $\tilde{\s}_j(k,\mu)=(ik)^m\tl_j^\infty(\mu)+o(k^m)$, which allows us to compute the real parts as
	\begin{equation}
		\Re\tilde{\s}_j(k,\mu)=(-1)^{\frac{m+1}{2}}k^m\b_j+o(k^m)
	\end{equation}
	Since $k^m$ changes sign we necessarily have to have $\b_j=0$, that is $\tl_j^\infty(\mu)$ are real numbers. For the second assertion, we Taylor expand the eigenvalues of $\tilde{S}(\eta,\mu)$ as
	\begin{equation}
		\tilde{\s}_j(\eta,\mu)=i^m\tl_j^\infty(\mu)+\d_\eta\tl_j(0,\mu)\eta+\cO(\eta^2)
	\end{equation}
	Passing back to $S(k,\mu)$, we find that
	\begin{equation}
		\Re \tl_j(k,\mu)=\Re\d_\eta\tl_j(0,\mu)k^{m-1}+\cO(k^{m-2})
	\end{equation}
	But we may compute $\d_\eta\tl_j(0,\e)=\ell_j^\infty \d_\eta\tilde{S}(0,\mu)r^\infty_j=(-1)^{\frac{m-1}{2}}\ell_j^\infty\cL_{m-1}(\mu)r^\infty_j$.
\end{proof}
In the Turing hypotheses, we make conditions on $\d_k^2\tl(k_*,0)$; so we seek an effective way to compute this quantity. We accomplish this with the following lemma.

\begin{lemma}\label{lem:EigenDeriv}
	Let $M(x)=\sum_{j=0}^{m}x^jM_j$ be a matrix function where each $M_j\in M_n(\CC)$. Suppose that at $x=0$, there is exactly one eigenvalue equal to 0 and that it is simple. Let $\l(x)$ be that eigenvalue and define left and right eigenvectors $\ell(x)$, $r(x)$ satisfying the normalization condition 
$\ell'(x)r(x)=\ell(x)r'(x)=0$ for each $x$, along with the usual $\ell(x)r(x)\equiv 1$. 
Define a projection $\Pi:=r(0)\ell(0)$. 
Then we have the formula for $\l''(0)$:
	\begin{equation}
		\l''(0)=2\left(\ell(0)M_2r(0)-\ell(0)M_1(I_n-\Pi)N(I_n-\Pi)M_1r(0) \right),
	\end{equation}
	where $N=\left((I_n-\Pi)M_0(I_n-\Pi)\right)^{-1}$.
\end{lemma}

\begin{proof}
We begin by looking at $\ell'(x)$. From standard matrix perturbation theory \cite{K}, 
	we know that $\ell$ is a smooth function in a neighborhood of 0. By our normalization conditions, $\ell'(0)\in(I_n-\Pi)\CC^n$ and similarly $r'(0)\in(I_n-\Pi)\CC^n$. As before, we try to compute $\ell'(0)$ by differentiating the eigenvalue equation and setting $x=0$, obtaining
	\begin{equation}
		\ell'(0)M_0+\ell(0)M_1=\l'(0)\ell(0).
	\end{equation}
	Applying $(I_n-\Pi)$ on the left, we find that
	\begin{equation}
		\ell'(0)M_0(I_n-\Pi)+\ell(0)M_1(I_n-\Pi)=0.
	\end{equation}
	Since $M_0$ is invertible on the invariant subspace $(I_n-\Pi)$, we can solve for $\ell'(0)$ as 
	\begin{equation}\label{eq:EllPrime}
		\ell'(0)=-\ell(0)M_1(I_n-\Pi)\left((I_n-\Pi)M_0(I_n-\Pi) \right)^{-1}
	\end{equation}
	Analogously, we find that
	\begin{equation}\label{eq:RPrime}
		r'(0)=-\left((I_n-\Pi)M_0(I_n-\Pi) \right)^{-1}(I_n-\Pi)M_1 r(0).
	\end{equation}
	In order to simplify notation, we define $N:=\left((I_n-\Pi)M_0(I_n-\Pi) \right)^{-1}$.

	As before, we compute $\l''(0)=2\ell(0)M_2r(0)+\ell'(0)M_1r(0)+\ell(0)M_1r'(0)$ and plugging in \eqref{eq:EllPrime} and \eqref{eq:RPrime}, we discover
	\begin{equation}
		\l''(0)=2\ell(0)M_2r(0)-\left[\ell(0)M_1(I_n-\Pi)NM_1r(0)+\ell(0)M_1N(I_n-\Pi)M_1 r(0) \right].
	\end{equation}
	To make this expression more symmetric, observe that by the functional calculus, that $N(I_n-\Pi)=(I_n-\Pi)N$; so we get our final expression
	\begin{equation}
		\l''(0)=2\left(\ell(0)M_2r(0)-\ell(0)M_1(I_n-\Pi)N(I_n-\Pi)M_1r(0) \right).
	\end{equation}
\end{proof}
The final preliminary result we will need is 
equivalence between translation invariant multilinear forms and multilinear multipliers. 
We recall from \cite{Mu} the proof of this fact in the periodic case.
\begin{proposition}\label{thm:multilin}
	Let $M:\cP(\TT)^k\to \cM(\TT)$ be multilinear where $\cP(\TT)$ is the space of trigonometric polynomials and $\cM(\TT)$ is the space of Borel measurable functions on the torus $\TT=(0,2\pi]$, 
and suppose that $M$ is translation invariant in the sense that for all translations $\tau_hf(x)=f(x-h)$ we have
	\begin{equation}
		\tau_hM(p_1,...,p_k)=M(\tau_hp_1,...,\tau_hp_k).
	\end{equation}
	Then there exists $\s:\ZZ^k\to\CC$ such that,
	denoting $e(lx)=e^{2\pi ilx}$,
	\begin{equation}
		M(e(l_1x),...,e(l_kx))=\s(l_1,...,l_k)e((l_1+...+l_k)x).
	\end{equation}
\end{proposition}

\begin{proof}
	The key identity underlying the proof is $f(x)=(\tau_{-x}f)(0)$. Applying this identity to $M(e(l_1\cdot),...,e(l_k\cdot))$ and using translation invariance, we get
	\begin{equation}
		M(e(l_1\cdot),...,e(l_k\cdot))(x)=(\tau_{-x}M(e(l_1\cdot),...,e(l_k\cdot)))(0)=M(\tau_{-x}e(l_1\cdot),...,\tau_{-x}e(l_k\cdot))(0)
	\end{equation}
	But $\tau_{-x}e(ly)=e^{2\pi il(x+y)}=e(ly)e(lx)$, so we get
	\begin{equation}
		M(e(l_1\cdot),...,e(l_k\cdot))(x)=M(e(l_1x)e(l_1\cdot),...,e(l_kx)e(l_k\cdot))(0)=(M(e(l_1\cdot),...,e(l_k\cdot))(0))e((l_1+...+l_k)x)
	\end{equation}
	Taking $\s(l_1,..,l_k)=M(e(l_1\cdot),...,e(l_k\cdot))(0)$ proves the theorem.
\end{proof}
\begin{remark}
	This proof easily generalizes to $\cP(\TT^d)^k\to\cM(\TT^d)$. We remark that a version of this theorem is also true for functions defined on $\RR$, but the proof is more difficult.
\end{remark}
\section{Multiscale Expansion}\label{sec:Multi}
In this section, we assume Turing Hypothesis \ref{hyp:Turing}. Let $\ell,r$ be the left and right eigenvectors associated to $\tl(k_*,0)$ of the matrix $S(k_*,0)$ and $\Pi=r\ell$. 
For $\mu=\eps^2$, we seek an approximate solution to \eqref{eq:MasterEqn} of form
\begin{equation}\label{eq:Ansatz}
\begin{split}
	U^{\e}(x,t)=\frac{1}{2}\e A(\Hx,\Ht)e^{i\xi}r+c.c.+\e^2\left(\Psi_0^2(\Hx,\Ht)+\frac{1}{2}\Psi_1^2(\Hx,\Ht)e^{i\xi}+c.c.+\frac{1}{2}\Psi_2^2(\Hx,\Ht)e^{2i\xi} \right)\\+\e^3\Psi_0^3(\Hx,\Ht)+\frac{1}{2}\e^3\sum_{j=1}^3\Psi_j^3(\Hx,\Ht)e^{ij\xi}+c.c.
\end{split}
\end{equation}
where $\xi=k_*(x-d_*t)$, $\Hx=\e(x-(d_*+\delta) t)$, $\Ht=\e^2t$, and $d_*,\delta\in\RR$ 
are as yet undetermined constant, that is consistent to $O(\e^3)$, with truncation error defined as
$$
	\cR:=U_t^\e-L(\mu)U^\e-\cN(U^\e)=O(\e^4).
$$
In \eqref{eq:Ansatz}, the subscript identifies the 
(discrete) Fourier mode and the superscript denotes the order of $\e$ at which the coefficient appears.

The rest of this section is devoted to the proof of the following theorem.

\begin{theorem}\label{3.1}
	For any sufficiently smooth $A$ satisfying the 
	complex Ginzburg-Landau equation \eqref{eq:cGL} on $0\leq \Ht\leq T$, there exists for $0\leq t\leq T/\eps^2$ an approximate solution of \eqref{eq:MasterEqn} of the form \eqref{eq:Ansatz} and some choice of smooth $\cA_1:=\ell\Psi_1^2$ that is consistent to order $O(\e^3)$ where 
	$d_*=-\frac{\Im\tl(k_*,0)}{k_*}$ and $d_*+ \delta=-\Im\d_k\tl(k_*,0)$. 
	(There is no uniqueness here, as $\Psi_1^2$ and $\Psi_1^3$ are not fully determined at this order.)
\end{theorem}

Suppose the scaling $L(\mu)=L(0)+\e^2\d_\mu L(0)+\cO(\e^4)$. We compute the derivatives of the Ansatz $U^{\e}$, where the slow variables have been suppressed for notational clarity.
\begin{equation}\label{eq:AnsatzDt}
	U^{\e}_t(x,t)=
	\frac{1}{2}(-ik_*d_*\e)Ae^{i\xi}r+\frac{1}{2}\e^2(-(d_*+\delta) A_{\Hx}e^{i\xi}r-ik_*d_*\Psi_1^2e^{i\xi})+\e^3\frac{1}{2}A_{\Ht}e^{i\xi}r
	+c.c.+ other,
\end{equation}
where $c.c.$ denotes complex conjugate and $other$ denotes omitted terms that turn out to be 
extraneous for the purpose of deriving amplitude equations.
Specifically, these terms are either in 
discrete Fourier modes that have uniquely determined correctors, i.e. every mode but $\pm1$, or they are order $\e^4$ or higher
whereas the complex Ginzburg-Landau equation appears as a compatibility condition at order $\cO(\eps^3)$.
Continuing, one can show by an inductive argument that
\begin{equation}\label{eq:AnsatzDxj}
\begin{split}
	\d_x^j U^{\e}(x,t) =
	%\left( 
	\frac{1}{2} \e (ik_*)^j Ae^{i\xi}r+\frac{1}{2} \e^2 (j(ik_*)^{j-1}A_{\Hx}e^{i\xi}r+(ik_*)^j\Psi_1^2e^{i\xi}+(2ik_*)^j\Psi_2^2e^{2i\xi})+
	%\right. 
	\\
	%\left. 
	+\frac{1}{2}\e^3(j(j-1)(ik_*)^{j-2}A_{\Hx\Hx}e^{i\xi}r+j(ik_*)^{j-1}\Psi_{1,\Hx}^2e^{i\xi} )+c.c. 
	%\right)
	+\cO(\e^4).
\end{split}
\end{equation}
Plugging this result into the formula for $L(0)U$, one finds that
\begin{equation}\label{eq:AnsatzL}
\begin{split}
	L(0)U^{\e}(x,t)=&
	\frac{1}{2}\e Ae^{i\xi}S(k_*,0)r+\e^2\cL_0(0)\Psi_0^2
	\\&+
	\frac{1}{2}\e^2(S(k_*,0)\Psi_1^2e^{i\xi}+S(2k_*,0)\Psi_2^2e^{2i\xi}-iA_{\Hx}e^{i\xi}\d_kS(k_*,0)r)+
	\\
	&+
	\frac{1}{2}\e^3(-A_{\Hx\Hx}e^{i\xi}\d_k^2S(k_*,0)r-i\d_kS(k_*,0)\Psi_{1,\Hx}^2e^{i\xi})+S(k_*,0)\Psi_1^3e^{i\xi}
	\\
	& +c.c.+other,
\end{split}
\end{equation}
where we've used $\d_k^lS(k,\mu)=\sum_{j=l}^m\binom{j}{l}i^l(ik)^{j-l}\cL_j(\mu)$. We 
next expand the nonlinearity in \eqref{eq:MasterEqn} into a Taylor series 
$$
\cN(U)=\cQ(U,U)+\cC(U,U,U)+\cO(|U|^4),
$$
where $\cQ$ is a bilinear form and $\cC$ is a trilinear form. 

Now we plug the Ansatz \eqref{eq:Ansatz} into \eqref{eq:MasterEqn} and collect terms of the form 
$c \e^Ne^{iM\xi}$ where $N\in\NN$ and $M\in\ZZ$, setting the resulting sums to zero.
For $\e e^{i\xi}$, we obtain
\begin{equation}\label{eq:e11}
	A\left[S(k_*,0)+id_*k_*\right]r=0 ,
\end{equation}
which can be solved for $d_*$ by
\begin{equation}\label{eq:dstar}
	d_*:=-\frac{\Im\tl(k_*,0)}{k_*}.
\end{equation}
For $\e^2 e^{0i\xi}$, we have
\begin{equation}\label{eq:e20}
	\cL_0(0)\Psi_0^2+\frac{1}{4}|A|^2\left(\cQ(r,\overline{r})+\cQ(\overline{r},r) \right)=0
\end{equation}
which we may solve for $\Psi_0^2$ using the Turing hypotheses as
\begin{equation}\label{eq:Psi0}
	\Psi_0^2=|A|^2\left(-\frac{1}{4}\cL_{0}(0)^{-1}\left[\cQ(r,\overline{r})+\cQ(\overline{r},r)\right] \right)=|A|^2v_0,
\end{equation}
where $v_0\in\RR^n$ is a known vector. 
This is to be expected:
$\Psi_0^2$ should be real valued by Fourier inversion and the formula explicitly confirms this regardless of whether or not $r$ is a real vector.

\begin{remark}
To see that $v_0\in\RR^n$, we may use the fact that $\cQ$ is built out of derivatives of $\cN$,
hence descends to a bilinear form $\cQ:\RR^n\times\RR^n\rightarrow\RR^n$. 
	Once $\cQ$ is a real bilinear form, we see that $\overline{\cQ(U,V)}=\cQ(\overline{U},\overline{V})$ by writing $\cQ$ as a direct sum of quadratic forms $\cQ_j:\RR^n\times\RR^n\rightarrow\RR$. Recall that all bilinear forms of the type $Q:\RR^n\times\RR^n\rightarrow\RR$ are given by
		$Q(U,V)=\sum_{i,j=0}^n U_iQ_{ij}V_j$,
	where $Q_{ij}$ is a unique real matrix. It is critical that $v_0$ be real, as otherwise our Ansatz 
	wouldn't be real-valued.
\end{remark}

We next explore $\e^2e^{2i\xi}$, where we find
\begin{equation}\label{eq:e22}
	\left(S(2k_*,0)+2ik_*d_* \right)\Psi_2^2+\frac{1}{2}A^2\cQ(r,r)=0.
\end{equation}
By the Turing hypotheses, $S(2k_*,0)$ has eigenvalues of negative real part; so we can invert $S(2k_*,0)+2ik_*d_*$ and find that
\begin{equation}\label{eq:Psi2}
	\Psi_2^2=-A^2\frac{1}{2}\left(S(2k_*,0)+2ik_*d_*\right)^{-1}\cQ(r,r)=A^2v_2,
\end{equation}
where $v_2\in\CC^n$ is a known vector.
\begin{remark}
	In reaction diffusion, with $n=2$, we have that $v_2$ is also real since $d_*=0$ and $S(2k_*,0)$ is a real matrix. Generically $v_2$ is not a real vector, unlike $v_0$.
\end{remark}
Finally, we look at $\e^2e^{i\xi}$. Here, we obtain the linear equation
\begin{equation}\label{eq:e21}
	\left(S(k_*,0)+ik_*d_* \right)\Psi_1^2+A_{\Hx}\left(-i\d_k S(k_*,0)+d_*+\delta \right)r=0.
\end{equation}
For this to be solvable, it is necessary that $\ell$\eqref{eq:e21} vanish. 
Computing this quantity, we obtain
\begin{equation}\label{eq:delta2}
	A_{\Hx}\ell\left(-i\d_k S(k_*,0)+d_*+\delta\right)r=0,
\end{equation}
or, using the fact that
$-i\ell\d_kS(k_*,0)r=-i\d_k\tl(k_*,0)=\Im\d_k\tl(k_*,0)$, 
\begin{equation}\label{eq:delta}
	\delta=-\Im\d_k\tl(k_*,0)-d_*.
\end{equation}
Writing $\Psi_1^2=\cA_1(\Hx,\Ht)r+\psi^{(1)}$ where $\psi^{(1)}\in(I_n-\Pi)\CC^n$, following the notation of Lemma \ref{lem:EigenDeriv}, we can then solve for $\psi^{(1)}$ as
\begin{equation}\label{eq:GeneralCasePsi1}
	\psi^{(1)}=iA_{\Hx}N(I_n-\Pi)\d_kS(k_*,0)r.
\end{equation}
Note that $\cA_1$ is free, hence at this level $\Psi_1$ is not completely determined. See Section \ref{s:allorders} to see how to determine $\cA_1$ via a compatibility condition at a higher order of $\e$.

We proceed now to the final mode of interest, 
$\e^3e^{i\xi}$. This gives us
\begin{equation}\label{eq:e31}
\begin{split}
	A_{\Ht}r-(d_*+\delta)\Psi_{1,\Hx}^2=(S(k_*,0)+id_*k_*)\Psi_1^3-i\d_kS(k_*,0)\Psi_{1,\Hx}^2-\\-\d_k^2S(k_*,0)A_{\Hx\Hx}r+\d_\mu S(k_*,0)Ar+|A|^2Av_3
\end{split}
\end{equation}
where we've simplified
the nonlinear expression using \eqref{eq:Psi0} and \eqref{eq:Psi2}, and the observation that the only nonlinear terms appearing are of the form (modulo permutations in the arguments) $\cQ(\Psi_0,Ar)$, $\cQ(\Psi_2,\overline{Ar})$, and $\cC(Ar,Ar,\overline{Ar})$. 
Here, $v_3\in\CC^n$ is an (in principle) known, constant vector.
We are interested in the solvability of \eqref{eq:e31}; thus, as in \eqref{eq:e21}, we apply $\ell$ to both sides
to obtain
\begin{equation}\label{eq:cgl1}
	A_{\Ht}=-i\ell\d_k S(k_*,0)\psi^{(1)}_{\Hx}-A_{\Hx\Hx}\ell\d_k^2S(k_*,0)r+\d_\mu\tl(k_*,0)A+\gamma|A|^2A,
\end{equation}
where $\gamma=\ell v_3\in\CC$ is a known constant. See \eqref{eq:generalgamma} in Lemma \ref{lem:GeneralGamma} for the formula for $\g$ in terms of spectral structure of $S(k_*,0)$ and Fr\'echet derivatives of $\cN$.
Plugging \eqref{eq:GeneralCasePsi1} into \eqref{eq:cgl1} then yields
\begin{equation}\label{eq:cgl2}
	A_{\Ht}=\left(\ell\d_kS(k_*,0)(I_n-\Pi)N(I_n-\Pi)\d_kS(k_*,0)r-\ell\d_k^2S(k_*,0)r \right)A_{\Hx\Hx}+\d_\mu\tl(k_*,0)A+\gamma|A|^2A.
\end{equation}
Applying Lemma \ref{lem:EigenDeriv} to the matrix function $S(k,0)+id_*kI_n$, we may reduce the equation to its final form
\begin{equation}\label{eq:cgl3}
	A_{\Ht}=-\frac{1}{2}\d_k^2\tl(k_*,0)A_{\Hx\Hx}+\d_\mu\tl(k_*,0)A+\gamma|A|^2A,
\end{equation}
namely, the complex Ginzburg-Landau equation \eqref{eq:cGL} of the introduction.\\

For $j\neq1$, $\Psi_j^3$ can be uniquely determined in terms of $A$ and $\cA_1$ in an entirely similar manner to $\Psi_0^2$ and $\Psi_2^2$.\\

The final point which we wish to mention is that each successive mode is resolved as a bounded function of previous modes and finitely many of their derivatives, hence, by induction, a bounded function of $A$, $\cA_1$, and their derivatives.

\begin{remark}\label{rem:MSEPseudodiff}
	The argument in this section will work if we merely assume that we are given a symbol $S(k,\mu)$ satisfying Turing hypotheses \ref{hyp:Turing}. The relevant modification is that \eqref{eq:AnsatzL} is replaced by the following formal expression coming from Taylor expanding the symbol at each frequency
	\begin{equation}
		L(k,\mu)U^\e(\Hx,\Ht,\xi)=L(k_*,0)U^\e-iL_k(k_*,0)\d_{\Hx}U^\e-\frac{1}{2}L_{kk}(k_*,0)\d_{\Hx}^2U^\e+\mu L_\mu(k_*,0)U^\e+h.o.t.
	\end{equation}
	where $L_k(k_*,0)$ has symbol $S_k(k_*,0)$, and a similar convention holds for the other operators appearing above. This comes from the observation that if $U$ is $\frac{1}{k}$ periodic, then $\widehat{\d_x}=ik\eta=ik_*\eta+i\e\o\eta=\widehat{\d_\xi}+i\e\o\eta$ where $k-k_*=\e\o$. But we also have by the chain rule that $\d_x=\d_\xi+\e\d_{\Hx}$, thus we have that $\widehat{\d_{\Hx}}=i\o\eta$. Note that one needs to assume bounds on $S(k,\mu)$ and its derivatives in order to make sense of the above formula on all periodic functions.
	However, here we do not require these, since our Ansatz is compactly supported in frequency.
\end{remark} 

\section{Lyapunov-Schmidt Reduction}\label{sec:LSRed}

In this section, we look for steady state periodic solutions to \eqref{eq:MasterEqn}, 
assuming as before Turing Hypotheses \ref{hyp:Turing}.
We change coordinates slightly, by taking $x=x-dt$ where $d$ is close to $d_*$ as defined in \eqref{eq:dstar}. 
We will also write $k=k_*+\kappa$ where $\kappa$ is a small number. In this section, $\mu$ will denote the bifurcation parameter and $\e$ will refer to a universal scaling parameter. We assume the scalings $\mu\sim\e^2$ and $\kappa\sim\e$. 

\subsection{Preliminary Estimates}
Here, we make a spectral assumption on $\cL_m(\mu)$.  This is to ensure that $L(\mu)^{\text{``-1''}}$ is a bounded operator from $L^2_{per}(\RR;\RR^n)$ to $H^m_{per}(\RR;\RR^n)$.
\begin{hypothesis}\label{hyp:Spectrum}
	If $m$ is odd, then $\sigma(\cL_m(\mu))\subset\RR\backslash\{0\}$. If $m$ is even, then the ellipticity condition in Proposition \ref{prop:NecSuf} holds, i.e. $\sigma((-1)^{m/2}\cL_m(\mu))\subset\{\lambda\in\CC:\ \Re(\lambda)<0 \}$.
\end{hypothesis}
Before we show boundedness of the inverse operator, we will prove the following technical lemma.
\begin{lemma}\label{lem:SVDNorm}
	Let $A\in M_n(\CC)$ be an invertible matrix. Then 
	\begin{equation}\label{eq:SVDNorm}
		||A^{-1}||=\frac{1}{\sigma_{min}(A)}
	\end{equation}
	where $\sigma_{min}(A)$ is the smallest singular value of $A$.
\end{lemma}
\begin{proof}
	Let $A=UDV^*$ be the singular value decomposition of $A$. Then $A^{-1}$ has singular value decomposition $A^{-1}=VD^{-1}U^*$, and so we can compute
	\begin{equation}\label{eq:InvNormCalc1}
		||A^{-1}||=\sup_{||x||=1}||A^{-1}x||=\sup_{||x||=||y||=1} |<A^{-1}x,y>|=\sup_{||x||=||y||=1}|<VD^{-1}U^*x,y>|,
	\end{equation}
	where $<x,y>$ is the usual inner product on $\CC^n$. Writing $x=Uu$ and $y=Vv$ 
	allows us to rewrite the last expression in \eqref{eq:InvNormCalc1} as
	\begin{equation}\label{eq:InvNormCalc2}
		\sup_{||x||=||y||=1}|<VD^{-1}U^*x,y>|=\sup_{||u||=||v||=1}|<D^{-1}U^*(Uu),V^*(Vv)>|=\sup_{||u||=||v||=1}|<D^{-1}u,v>|.
	\end{equation}
	The claim then follows from the fact that the norm of a diagonal matrix is the largest element on the diagonal.
\end{proof}

%OLD VERSION:
%\begin{proposition}\label{prop:LinBdd}
	%Assume Hypotheses \ref{hyp:Turing} and \ref{hyp:Spectrum}. Let $P$ be the projection onto the neutral mode. 
	%Then
		%\begin{equation}
			%T(k,\mu):=\left((I-P)L(k,\mu)(I-P) \right)^{-1}:(I-P)L^2_{per}(\RR;\RR^n)\rightarrow (I-P)H^m_{per}(\RR;\RR^n)
		%\end{equation}
		%is a bounded operator for $\kappa$ and $\mu$ sufficiently small, with bounds independent of $\kappa$ and $\mu$. More generally, one has for all $s\in\RR$ that $T:H^s_{per}(\RR;\RR^n)\rightarrow H^{s+m}_{per}(\RR;\RR^n)$ in a bounded manner, with bounds only depending on $s$, $\kappa$ and $\mu$.
%\end{proposition

\begin{proposition}\label{prop:LinBdd}
	Assume Hypotheses \ref{hyp:Turing} and \ref{hyp:Spectrum}. Let $P$ be the projection onto the neutral mode, i.e. the bifurcating eigenvalue $\tl$ in Hypothesis \ref{hyp:Turing}, defined by
	$$
		PU(\xi):=r\ell \widehat{U}(1)e^{i\xi}+c.c.,
	$$
	where $r$ and $\ell$ are the right/left (resp.) eigenvectors of $S(k_*,0)$ associated to the neutral eigenvalue $\tl(k_*,0)$.
	Then
		\begin{equation}
			T(k,\mu):=\left((I-P)L(k,\mu)(I-P) \right)^{-1}:(I-P)L^2_{per}(\RR;\RR^n)\rightarrow (I-P)H^m_{per}(\RR;\RR^n)
		\end{equation}
		is a bounded operator for $\kappa$ and $\mu$ sufficiently small, with bounds independent of $\kappa$ and $\mu$. More generally, one has for all $s\in\RR$ that $T:H^s_{per}(\RR;\RR^n)\rightarrow H^{s+m}_{per}(\RR;\RR^n)$ in a bounded manner, with bounds only depending on $s$, $\kappa$ and $\mu$.
\end{proposition}

\begin{proof}
	Observe that $(I-P)L(k,\mu)(I-P)$ is a Fourier multiplier operator with multiplier
	%$m(\eta,k,\mu)$
	\begin{equation}
		m(\eta, k,\mu)=\begin{cases}
			S(\eta k,\mu), \text{ for } \eta k\not=\pm k_*,\\
			(I-\Pi)S(k_*,\mu)(I-\Pi), \text{ for } \eta k=k_*,\\
			\overline{(I-\Pi)S(k_*,\mu)(I-\Pi)}, \text{ for } \eta k=-k_*.
		\end{cases}
	\end{equation}
	Hence, the inverse operator has multiplier $m^{-1}(\eta,k,\mu)$
	\begin{equation}
		m^{-1}(\eta,k,\mu)=\begin{cases}
			S(\eta k,\mu)^{-1}, \text{ for } \eta k\not=\pm k_*,\\
			\left((I-\Pi)S(k_*,\mu)(I-\Pi) \right)^{-1}, \text{ for } \eta k=k_*,\\
			\left(\overline{(I-\Pi)S(k_*,\mu)(I-\Pi)}\right)^{-1}, \text{ for } \eta k=-k_*.
		\end{cases}
	\end{equation}

	By Lemma \ref{lem:SVDNorm}, it suffices to show that there exist 
	$\kappa_0>0$, $\mu_0>0$, and $\eta_0>0$ such that 
	\begin{equation}\label{eq:HighFreq1}
		\inf_{|\mu|\leq\mu_0}\inf_{|\kappa|\leq\kappa_0}\inf_{|\eta|\geq \eta_0}\sigma_{min}(m(\eta, k,\mu))|\eta|^{-m}>c
	\end{equation}
	for some $c>0$ and
	\begin{equation}\label{eq:LowFreq1}
		\inf_{|\mu|\leq\mu_0}\inf_{|\kappa|\leq\kappa_0}\min_{|\eta|\leq \eta_0}\sigma_{min}(m(\eta, k,\mu))>0.
	\end{equation}
	For, supposing both \eqref{eq:HighFreq1} and \eqref{eq:LowFreq1}, we compute $||TU||_{H^m_{per}(\RR;\RR^n)}$ as
	\begin{equation}\label{eq:TBdd1}
		||T(k,\mu)U||_{H^m_{per}(\RR;\RR^n)}^2=
		\sum_{\eta\in\ZZ}(1+|\eta|^2)^m |m^{-1}(\eta k,\mu)\hat{U}(\eta)|^2.
	\end{equation}
	We can be generous in \eqref{eq:TBdd1} and use the norm bound on $m^{-1}(\eta,k,\mu)$ to write
	\begin{equation}\label{eq:TBdd2}
	 ||T(k,\mu)U||_{H^m_{per}(\RR;\RR^n)}^2\leq \sum_{\eta\in\ZZ}(1+|\eta|^2)^m ||m^{-1}(\eta, k,\mu)||^2\cdot|\hat{U}(\eta)|^2.
	\end{equation}

	We split the above into two sums, the first where $|\eta|\leq\eta_0$ and the second where $|\eta|>\eta_0$.
	On the first sum, we use \eqref{eq:LowFreq1} in combination with Lemma \ref{lem:SVDNorm} to see that $||m^{-1}(\eta,k,\mu)||\leq C$ where $C>0$ is a fixed constant, moreover we can be generous and bound $(1+|\eta|^2)^m$ by $(1+|\eta_0|^2)^m$. For the second sum, we use \eqref{eq:HighFreq1} and the lemma to observe that $||m^{-1}(\eta,k,\mu)||\leq\frac{1}{c|\eta|^m}$. This lets us bound \eqref{eq:TBdd2} by
	\begin{equation}\label{eq:TBdd3}
		||T(k,\mu)U||_{H^m_{per}}(\RR;\RR^n)\leq C^2(1+|\eta_0|^2)^m\left(\sum_{|\eta|\leq\eta_0}|\hat{U}(\eta)|^2 \right)+\left( \sum_{|\eta|>\eta_0}(1+|\eta|^2)^m \frac{1}{c^2|\eta|^{2m}}|\hat{U}(\eta)|^2\right).
	\end{equation}

	To complete this part of the argument, it is a simple computation to show that $\frac{(1+|\eta|^2)^m}{c^2|\eta|^{2m}}$ is a bounded function of $\eta$ for $\eta$ large enough so that we get
	\begin{equation}\label{eq:TBdd4}
		||T(k,\mu)U||_{H^m_{per}}(\RR;\RR^n)\leq C^2\sum_{\eta\in\ZZ}|\hat{U}(\eta)|^2=C^2||U||_{L^2_{per}(\RR;\RR^n)}
	\end{equation}
	by Plancherel's theorem. Moreover, the constant doesn't depend on $\mu$ or $\kappa$ provided that they are sufficiently small.

	It remains to prove that
	\eqref{eq:HighFreq1} and \eqref{eq:LowFreq1} can be arranged for $\mu$ and $\kappa$ sufficiently small. Starting with \eqref{eq:HighFreq1}, we claim that $\sigma(\eta,k,\mu)$ is a singular value of $S(\eta k,\mu)$ if and only if $\frac{1}{(\eta k)^m}\sigma(\eta,k,\mu)$ is a singular value of $\tilde{S}(\frac{1}{\eta k},\mu)$ as it appears in \eqref{eq:TildeS}. This follows from the spectral mapping theorem and the observation that
	\begin{equation}
		\overline{S(\eta k,\mu)}S(\eta k,\mu)=(\eta k)^{2m}\overline{\tilde{S}\left(\frac{1}{\eta k},\mu\right)}\tilde{S}\left(\frac{1}{\eta k},\mu\right).
	\end{equation}

	By continuity of $\sigma_{min}$, there exist 
	$\kappa_0$, $\mu_0$, and $\tilde{k}_0$ such that for $|\tilde{k}|\geq \tilde{k}_0$ we have that
	\begin{equation}
		\inf_{|\kappa|\leq\kappa_0}\inf_{|\mu|\leq\mu_0}\inf_{|\tilde{k}|\geq\tilde{k}_0}|\sigma_{min}(\tilde{S}(\frac{1}{\tilde{k}},\mu))-\sigma_{min}(\cL_m(\mu))|\geq\frac{1}{2}\sigma_{min}(\cL_m(\mu)).
	\end{equation}
	The triangle inequality gives us
	\begin{equation}
		\inf_{|\kappa|\leq\kappa_0}\inf_{|\mu|\leq\mu_0}\inf_{|\tilde{k}|\geq\tilde{k}_0}|\sigma_{min}(\tilde{S}(\frac{1}{\tilde{k}},\mu))\geq\frac{1}{2}\sigma_{min}(\cL_m(\mu)).
	\end{equation}
	But, $\sigma_{min}(\tilde{S}(\frac{1}{\eta k},\mu))=(\eta k)^{-m}\sigma_{min}(S(\eta k,\mu))$. Defining $\eta_0:=\frac{\tilde{k}_0}{k^m}$ and using the scaling gives
	\begin{equation}\label{eq:HighFreq2}
		\inf_{|\kappa|\leq\kappa_0}\inf_{|\mu|\leq\mu_0}\inf_{|\eta|\geq\eta_0}\sigma_{min}(S(k,\mu))|\eta|^{-m}\geq\frac{1}{2}\sigma_{min}(\cL_m(\mu))k^m.
	\end{equation}
	This proves \eqref{eq:HighFreq1} provided $\kappa_0<k_*$. 
	
	For \eqref{eq:LowFreq1}, we fix the $\eta_0$ from \eqref{eq:HighFreq1}. For $\eta=0$, we see that $\sigma_{min}(S(0,\mu))=\sigma_{min}(\cL_0(\mu))$ which we can uniformly bound from below by a constant for $\mu_0$ small enough. For $2\leq|\eta|\leq\eta_0$, we see that for $\kappa_0$ and $\mu_0$ small enough because $|\eta k-\eta k_*|=|\eta\kappa|\leq \eta_0\kappa_0$, hence we can uniformly bound $||S(\eta k,\mu)^{-1}||$ by continuity of the inverse map. The uniformity follows because we can ensure that $S(2k,\mu),...,S(\eta_0 k,\mu)$ is close to $S(2k_*,\mu),...,S(\eta_0k_*,\mu)$ uniformly in $\eta$ by taking $\kappa_0$ small enough. For $\eta=\pm1$, we apply the preceding observation to $(I_n-\Pi)S(k,\mu)(I_n-\Pi)$. This procedure gives us finitely many $\kappa_0$'s and $\mu_0$'s, hence we may take the minimum and complete the proof.

	The claim for all $s\in\RR$ follows in a similar manner.
\end{proof}
We note that both the Turing hypotheses and the ellipticity assumptions \ref{hyp:Spectrum} are, outside of the special case $m=1$, stable under Galilean coordinate changes.
\subsection{The Reduction Procedure}
Define projectors
\begin{align}\label{eq:Projectors}
V=PU&:=\Pi \hat{U}(1)e^{i\xi}+c.c.\\
W=QU&:=(I_n-\Pi)\hat{U}(1)e^{i\xi}+c.c.\\
X=RU&:=\sum_{l\not=\pm1}\hat{U}(l)e^{il\xi}
\end{align}
Note that the linear operator $L(k,\mu)$ is constant coefficient, so $RL(k,\mu)=L(k,\mu)R$ for every $k$ and $\mu$.
\begin{remark}
	% This 
	The above is a slight abuse of notation. Technically, there should be 3 pairs of projectors defined by the formulas provided in \eqref{eq:Projectors}, because one should be acting on $H^m_{per}(\RR;\RR^n)$ and the other acting on $L^2_{per}(\RR;\RR^n)$. Since they're defined by the same formula, we will denote them by the same letter.

	As another remark, it is unnecessary to split $Q$ and $R$. The reduction can be carried out entirely using the projectors $P$ and $I-P$. Here, they are split in order to highlight the differences between the $\pm1$ Fourier modes and the other Fourier modes. At some level, a splitting of this type is required; if only to get the leading order behavior of the $\pm 1$ Fourier mode.
\end{remark}
We have that $L(k,\mu)U+\cN(U)=0$ if and only if 
$$
P\left(L(k,\mu)U+\cN(U) \right)=Q\left(L(k,\mu)U+\cN(U) \right)=R\left(L(k,\mu)U+\cN(U)\right)=0.
$$
Expanding, 
%this out 
and using the commutation relation between $R$ and $L(k,\mu)$, we find
\begin{equation}\label{eq:XEqn1}
	L(k,\mu)X+kdX_\xi+R\cN(V+W+X)=0.
\end{equation}
Here,
$R\left(L(k,\mu)+dk\d_\xi\right)$ is an invertible operator for $k$ close enough to $k_*$ and $\mu$ close enough to 0, 
moreover the inverse is a bounded operator by Proposition \ref{prop:LinBdd}, so by the implicit function theorem,
\begin{equation}\label{eq:XEqn2}
	X=\Psi(V,W;\mu,k,d)=\cO(|V|^2,|W|^2).
\end{equation}
We have, further, 
that $\Psi$ retains the $SO(2)$ invariance of \eqref{eq:MasterEqn}. We record this observation in Fourier space as
\begin{equation}
	\hat{\Psi}(e^{i\xi_0}\hat{V},e^{i\xi_0}\hat{W};\mu,k,d)=e^{i\xi_0}\hat{\Psi}(\hat{V},\hat{W};\mu,k,d).
\end{equation}

\begin{remark}\label{wrinklermk}
	There is a slight technical wrinkle if $m=1$; here we may lose boundedness of the inverse operator 
	if $\pm d\in\sigma(\cL_1(\mu))$, i.e., the critical wave speed $d_*$ is
	a ``natural,'' or characteristic, speed of the linear operator, 
	and so this must be assumed not to happen.
	This is none other than the usual Turing hypothesis at $k=0$,
	transported to the natural rest frame of bifurcating waves.
\end{remark}

Before we look at the equations of $Q$ and $P$, we need to expand $L(k,\mu)+dk\d_\xi$ into a more workable form. Since we're only looking at $P$ and $Q$, it suffices to understand $S(k,\mu)+idkI_n$. Now we can write
\begin{equation}
	S(k,\mu)=\sum_{j=0}^m\frac{1}{j!}\d_k^j S(k_*,\mu)\kappa^j
\end{equation}
because for each fixed $\mu$, $S(k_*+\kappa,\mu)$ is a polynomial in $\kappa$. Next, we Taylor expand with respect to $\mu$ to find
\begin{equation}\label{eq:SymbolTaylor1}
	S(k,\mu)=\sum_{j=0}^m\frac{1}{j!}\d_k^jS(k_*,0)\kappa^j+\sum_{j=0}^m\frac{1}{j!}\d_\mu\d_k^jS(k_*,0)\kappa^j\mu+\cO(\mu^2).
\end{equation}

Now, we expect that $|V|\sim\e$, so ideally we will be able to ignore all terms of order at least $\e^3$. Briefly, the reason for this scaling is that we're aiming for an equation of the form $\tilde{\tl}(\kappa,\mu)+n(|V|^2)=0$ where $\tilde{\tl}(\kappa,\mu)=\Re\d_\mu\tl(k_*,0)\mu+\frac{1}{2}\Re\d_k^2\tl(k_*,0)\kappa^2$, and the scaling provided above is precisely the one where each term could be comparable in size, though it may happen in special circumstances that the nonlinearity is of a smaller order.
 With this in mind, we record the only important terms in \eqref{eq:SymbolTaylor1} in the 
 equation
\begin{equation}\label{eq:SymbolTaylor2}
	S(k,\mu)=S(k_*,0)+\d_kS(k_*,0)\kappa+\frac{1}{2}\d_k^2S(k_*,0)\kappa^2+\d_\mu S(k_*,0)\mu+\cO(\mu^2,\kappa^3,\mu\kappa).
\end{equation}
This allows us to compute the commutation relations between $P$, $Q$ and $L(k,\mu)$. 

First, we look at $QL(k,\mu)$ using
$$
\begin{aligned}
	QL(k,\mu)U&=(I_n-\Pi)S(k,\mu)\hat{U}(1)e^{i\xi}+c.c.\\
	&=(I_n-\Pi)\left(S(k_*,0)+\d_k S(k_*,0)\kappa+\cO(\mu,\kappa^2)\right)\hat{U}(1)e^{i\xi}+c.c.
	\end{aligned}
	$$
Since $\Pi$ commutes with $S(k_*,0)$, we find that
$$
	\begin{aligned}
QL(k,\mu)U&=\left(S(k_*,0)(I_n-\Pi)\hat{U}(1)e^{i\xi}+\kappa (I_n-\Pi)S_k(k_*,0)\Pi\hat{U}(1)e^{i\xi}\right)\\
		&\quad +c.c.+\cO(\mu U,\kappa QU,\kappa^2U).
	\end{aligned}
	$$
Hence, applying $Q$ to Eq. \eqref{eq:MasterEqn} on the left gives 
\ba\label{eq:QLEqn3}
	\big((S(k_*,0)+id_*k_*)(I_n-\Pi)\hat{U}e^{i\xi}&+\kappa (I_n-\Pi)S_k(k_*,0)\a e^{i\xi}r+\cO(\mu U,\kappa W,\kappa^2 U)\big)
	\\
	&\quad +c.c.+Q\cN(V+W+\Psi)=0,
\ea
 where $V=\frac{1}{2}\a e^{i\xi}r+c.c.$ for some scalar $\a\in\CC$. %Scaling has been corrected, it should have been 1/2\a e^{i\xi} to match the CGL ansatz on the nose; otherwise everything will be off by powers of 2. This isn't an issue for Brusselator b/c cos and sin carry the needed 1/2. -A
 In this equation, we have enforced 
the scaling $\delta\sim\kappa$; the reason for this scaling will become apparent later on. 
We may solve this equation using the implicit function theorem (observing that $(S(k_*,0)+id_*k_*)(I-\Pi_n)$ 
is invertible since all eigenvalues have nonzero real part), obtaining
\begin{equation}\label{eq:ExactPsi1}
	\psi_1:=(I_n-\Pi)\hat{U}(1)=-\frac{1}{2}\kappa(I_n-\Pi)N(I_n-\Pi)\a r+\cO(\mu U,\kappa^2 U,|V|^2).
\end{equation}

\begin{remark}
	To connect this result to the multiscale expansion, note that the identifications $\a\leftrightarrow A$ and $i\kappa\leftrightarrow\d_{\Hx}$ allow us to conclude that 
	$\psi^{(1)}$ in \eqref{eq:GeneralCasePsi1} is in fact the linearization of 
	$\psi_1$ in \eqref{eq:ExactPsi1}, up to a factor of $\frac{1}{2}$. %\Psi_1 in the Ansatz has a factor of 1/2 that the \psi here doesn't have
\end{remark}

We conclude that $W=\Phi(V;\mu,k,d)=\cO(\kappa V,\mu V, |V|^2)$, moreover it inherits the $SO(2)$ invariance in the same way that $\Psi$ does.

Finally, we look at the equation for $P$. Here we find
\begin{equation}\label{eq:PEqn1}
	PL(k,\mu)(V+\Phi+\Psi)+idkV+P\cN(V+\Phi+\Psi)=0.
\end{equation}
Applying \eqref{eq:SymbolTaylor2}, we discover that
\begin{equation}
	P\left(S(k_*,0)+\kappa S_k(k_*,0)+\frac{1}{2}\kappa^2S_{kk}(k_*,0)+\d_\mu S(k_*,0)\mu\right)(V+\Phi)+idkV+P\cN(V+\Phi+\Psi)=0.
\end{equation}
Expanding somewhat further, we find
\begin{equation}\label{eq:PEqn2}
\begin{split}
	\tl(k_*,0)V+idkV+\d_k\l(k_*,0)\kappa V+\frac{1}{2}\kappa^2PS_{kk}(k_*,0)V+\cO(\kappa^2\Phi)+\d_\mu\tl(k_*,0)\mu V+\cO(\mu\Phi)\\+\kappa PS_k(k_*,0)\Phi+P\cN(V+\Phi+\Psi)=0.
\end{split}
\end{equation}
We can simplify a bit by applying \eqref{eq:ExactPsi1}, whence we obtain
\begin{equation}\label{eq:PEqn3}
\begin{split}
	\left(\tl(k_*,0)+idk+\d_k\tl(k_*,0)\kappa+\d_\mu\tl(k_*,0)\mu+\frac{1}{2}\kappa^2\ell S_{kk}(k_*,0)r-\right. \\ 
	\left. \kappa^2\ell S_k(k_*,0)(I_n-\Pi)N(I_n-\Pi)S_k(k_*,0)r \right)\frac{1}{2}\a e^{i\xi}r\\
	+c.c.+\cO(\mu\Phi,\kappa^2\Phi,\mu\kappa U,\mu^2 U,\kappa^3 U,\kappa|V|^2)+P\cN(V+\Phi+\Psi)=0.
\end{split}
\end{equation}

Now we may apply Lemma \ref{lem:EigenDeriv} to the matrix function $S(k_*+\kappa,0)+id_*(k_*+\kappa)$ with $M_j=\frac{1}{j!}\d_k^j(S(k,0)+id_*k)|_{k=k_*}$ and $x=\kappa$. In this notation, we can rewrite the relevant terms in \eqref{eq:PEqn3} as
\begin{equation}\label{eq:PEqn4}
\begin{split}
\left(\tl(k_*,0)+id_*k_*+\d_k\tl(k_*,0)\kappa+\frac{1}{2}\d_k^2\tl(k_*,0)\kappa^2+i(d_*\kappa+\delta k_*+\delta\kappa)\right)\frac{1}{2}\a e^{i\xi}r+c.c.+\\
+\cO(\mu\Phi,\kappa^2\Phi,\mu\kappa U,\mu^2 U,\kappa^3 U,\kappa|V|^2)+P\cN(V+\Phi+\Psi)=0.
\end{split}
\end{equation}
This equation can be solved if and only if the coefficients of $e^{i\xi}$ and $e^{-i\xi}$ vanish separately. So let's consider the coefficient of $e^{i\xi}$ by itself, where every term in \eqref{eq:PEqn4} is parallel to $r$, hence it is a scalar equation in disguise. Making these reductions, we get the 
%following 
equation
\begin{equation}\label{eq:Epseqn1}
\begin{split}
	\left(\d_k\tl(k_*,0)\kappa+\frac{1}{2}\d_k^2\tl(k_*,0)\kappa^2+\d_\mu\tl(k_*,0)\mu+i(d_*\kappa+\delta k_*+\delta\kappa) \right)\frac{1}{2}\a+\\+\cO(\mu\kappa\a,\kappa^3\a,\mu^2\a,\kappa\a^2)+\tilde{\cN}(\a;\mu,k,d)=0.
\end{split}
\end{equation}

We may then divide by $\frac{1}{2}\a$ in \eqref{eq:Epseqn1} to remove the trivial solutions, obtaining
\ba\label{eq:Epseqn2}
	\frac{1}{2}\big(\d_k\tl(k_*,0)\kappa+ & \frac{1}{2}\d_k^2\tl(k_*,0)\kappa^2+\d_\mu\tl(k_*,0)\mu+i(d_*\kappa+\delta k_*+\delta\kappa) \big)\\
	&\quad +
	\cO(\mu\kappa,\kappa^3,\mu^2,\kappa\a)+n(\a;\mu,k,d)=0.
	\ea
Exploiting $SO(2)$ invariance we see that $n(\a;\mu,k,d)=n(|\a|^2;\mu,k,d)$ 
because $\tilde{\cN}(e^{i\xi_0}\a;\mu,k,d)=e^{i\xi_0}\tilde{\cN}(\a;\mu,k,d)$ and $\tilde{\cN}$ 
inherits the $SO(2)$ invariance of $\cN$, $\Phi$ and $\Psi$. 

Evidently, \eqref{eq:Epseqn2} can be solved if and only if both its real and imaginary parts vanish simultaneously. First, we look at the real part,
\begin{equation}\label{eq:ReEpseqn2}
	\Re\d_\mu\tl(k_*,0)\mu+\frac{1}{2}\kappa^2\Re\d_k^2\tl(k_*,0)+\cO(\mu\kappa,\kappa^3,\mu^2,\kappa|\a|)+\Re n(|\a|^2;\mu,k,d)=0.
\end{equation}
To rewrite the above in a more usable form, we write $\mu=\e^2\tilde{\mu}$, $\kappa=\tilde{\kappa}\e$, and $\a=\tilde{\a}\e$ in order to isolate $\tilde{\mu}$ and remove the trivial solution. In this scaling, we can rewrite \eqref{eq:ReEpseqn2} as
\begin{equation}\label{eq:ReEpseqn3}
	\e^2\left(\Re\d_\mu\tl(k_*,0)\tilde{\mu}+\frac{1}{2}\tilde{\kappa}^2\Re\d_k^2\tl(k_*,0) \right)+\cO(\e^3)+\e^2\Re\tilde{n}(|\tilde{\a}|^2;\mu,k,d)=0.
\end{equation}
Taylor expanding $\tilde{n}(|\tilde{\a}|^2;\mu,k,d)$ as $\g_{LS}|\tilde{\a}|^2+\cO(\e)$, we see that the sign of $\Re\g_{LS}$ determines whether the Turing bifurcation is subcritical or supercritical. In particular, if $\Re\g_{LS}<0$ then 
it is supercritical and if $\Re\g_{LS}>0$ then it is subcritical. 
This follows from solving for $\tilde{\a}=\tilde{\a}(\e,\tilde{\kappa})$ in
\begin{equation}\label{eq:alphaeqn}
	|\tilde{\a}|^2=\frac{-\frac{1}{2}\tilde{\kappa}^2\d_k^2\tl(k_*,0)
	-\Re\d_\mu\tl(k_*,0)\tilde{\mu} }{\Re\g_{LS}}+\cO(\e).
\end{equation}

It's clear that this equation has a unique positive solution for $\e$ sufficiently small if it has a solution at all. 
Note that if this equation has a solution for $\tilde{\a}$ at $\tilde{\kappa}=0$, we need $\tilde{\mu}$ and $\Re\g_{LS}$ to have opposite signs, which is why it's supercritical when $\Re\g_{LS}<0$.  We will show in Subsection \ref{subsec:TwoGammas} that $\g_{LS}$ is the same $\g$ as the one in Section \ref{sec:Multi}, so that this temporary subscript can be dropped. We know by \eqref{eq:QLEqn3} and \eqref{eq:XEqn2} that our desired solution $U$ admits the expansion
\begin{equation}
	U=\frac{1}{2}\a e^{i\xi}r+c.c.+\Phi(\a;k,\mu,d)+\Psi(\a;k,\mu,d)
\end{equation}
with $\Psi=\cO(\a^2)$ and $\Phi=\cO(\e\a,\a^2)$. Since we've adopted the scaling $\a=\cO(\e)$, this in turn implies that $U=\frac{1}{2}\a e^{i\xi}r+c.c.+\cO(\e^2)$ with $\a=\e\tilde{\a}$ and $\tilde{\a}$ taken to be the unique positive solution of \eqref{eq:alphaeqn}.

For the imaginary part of \eqref{eq:Epseqn2}, we find
\begin{equation}
	\Im\d_k\tl(k_*,0)\kappa+\frac{1}{2}\Im\d_k^2\tl(k_*,0)\kappa^2+\Im\d_\mu\tl(k_*,0)\mu+d_*\kappa+\delta k_*+\delta\kappa+\Im n(|\a|^2;\mu,d,k)+\cO(\e^3)=0,
\end{equation}
which is solvable for $\delta$ as a function of $\kappa,\a$ by the implicit function theorem since $k_*\not=0$. 
Note that to lowest order
\begin{equation}
	k_*\delta=-\Im\d_k\tl(k_*,0)\kappa-d_*\kappa.
\end{equation}
If we divide by $\kappa$ and let $\tilde{\delta}=\frac{\delta}{\kappa}=\cO(1)$, then we find that
\begin{equation}
	k_*\tilde{\delta}=-\Im\d_k\tl(k_*,0)-d_*.
\end{equation}
Comparing with \eqref{eq:delta}, we see that the $\delta$ in \eqref{eq:delta} is in fact $k_*\tilde{\delta}$ to lowest order.

Overall, in this section we have shown the following theorem.

\begin{theorem}\label{thm:LSReduction}
	There exists an $\eps_0>0$ and a $\nu_0>0$ so that for all $\e<\e_0$ all $\tilde{\kappa}^2\leq(1-\nu_0\e)\tilde{\kappa_E}^2$, where $\tilde\kappa_E^2$ is given by \eqref{krange}, there is a unique solution $\tilde{u}_{\e,\tilde{\kappa}}\in H^m_{per}([0,2\pi];\RR^n)$, with $\a>0$ at $\tilde{\kappa}=0$, and $d\in \RR$ satisfying \eqref{eq:MasterEqn} with $k=k_*+\e\tilde\kappa$. Moreover, $\tilde{u}_{\e,\tilde{\kappa_0}}$ admits the expansion
	\begin{equation}
		\tilde{u}_{\e,\tilde{\kappa}}=\frac{1}{2}\e\sqrt{\frac{-\frac{1}{2}\tilde{\kappa}^2\d_k^2\tl(k_*,0)
				-\Re\d_\mu\tl(k_*,0)\tilde{\mu} }{\Re\g_{LS}}}e^{i\xi}r+c.c.+\cO(\e^2).
	\end{equation}
\end{theorem}

\begin{remark}\label{rem:LSPseudodiff}
	In this section, we only really used that there was a symbol satisfying the Turing hypotheses and certain bounds on the eigenvalues and singular values. In particular, we need the existence of $K$ a compact neighborhood with nonempty interior of $(k_*,0)$ such that outside of $K$ there exists universal constants $s\geq 1$, $c>0$, and $\Lambda_0>0$ satisfying
	\begin{equation}
		c^{-1}(1+|k|^2)^\frac{s}{2}\leq \s_{\min}(S(k,\mu))\leq \s_{\max}(S(k,\mu))\leq c(1+|k|^2)^\frac{s}{2}
	\end{equation}
	for all $(k,\mu)\not\in K$ and that
	\begin{equation}
		\max\{\Re \tl_j(k,\mu) \}\leq -\Lambda_0
	\end{equation}
	for all $(k,\mu)\not\in K$. The first condition gives ellipticity and the second is a form of spectral stability. We need these because in the case of a  general symbol, these bounds are not automatic, whereas
	in the case of differential operators they follow from the sufficient conditions in Hypothesis \ref{hyp:Spectrum}.
\end{remark}

\br\label{nonTuringrmk}
 %It is interesting to note that in 
In deriving the results of the previous two sections,
we have used nowhere the fact that $\Re \sigma(L(0))<0$ away from the critical
mode at $k=k_*$, $\lambda=i k_* d_*$, but only the (implied) properties that
(i) except for this critical mode, $\Re \sigma(L(0))\neq 0$ on the lattice $k_* \ZZ$ (i.e., nonresonance), 
and (ii) $\tilde \lambda(0,k)$ is stationary at $k=k_*$.
That is, (cGL) can also well-describe ``secondary'' Turing bifurcations, defined as local changes in stability
of the eigenvalues of the symbol $S(0,k)$.
\er

%\section{More general nonlinearities: A case study}\label{sec:MGN}
\section{More general nonlinearities}\label{sec:MGN}
%We now discuss the extension to more general types of nonlinearities, starting with quasilinear ones.
In this section, we will take for simplicity $m=2$ where $m$ is the order of the system in \eqref{eq:MasterEqn}, 
and study the general quasilinear system
\begin{equation}\label{eq:MGNMasterEqn}
	u_t=(h(u;\mu)u_x)_x+f(u;\mu)_x+g(u;\mu),
\end{equation}
where $f,g:\RR^n\times\RR\rightarrow\RR^n$ are $C^{\infty}$ and $h:\RR^n\times\RR\rightarrow M_n(\RR)$ is $C^{\infty}$.  
Other values of $m$ follow by an entirely similar argument.

Suppose for the moment that $u(x,t)$ is an $H^1(\RR_t:H^2_{per}(\RR_x;\RR^n))$ solution to \eqref{eq:MGNMasterEqn}. Then by Sobolev embedding, for each fixed $t$, $u(x,t)\in C^1_{per}(\RR_x;\RR^n)$. Hence the quantity inside the first bracket is in $H^1_{per}(\RR_x;\RR^n)$ and the quantity in the second bracket is in $C^1_{per}(\RR:\RR^n)$.

\begin{remark}
	For the above bounds, it is important that we work in one spatial variable $x$.
	For higher dimensions, we would need to work in a higher-regularity space $H^s$, with
	$s$ chosen, as is standard, according to Sobolev embedding requirements.
	However the approach would be essentially the same.
\end{remark}

 Suppose that there is a smooth function $u_*=u_*(\mu)$ satisfying $g(u_*(\mu);\mu)=0$,
 and suppose that the linearized operator
 \begin{equation}\label{eq:MGNLinOp}
	 L(\mu):=h(u_*(\mu);\mu)\d_x^2+f_u(u_*(\mu);\mu)\d_x+g_u(u_*(\mu);\mu)
 \end{equation}
 admits a Turing bifurcation. Let $\cL_2(\mu):=h(u_*(\mu);\mu)$, $\cL_1(\mu):=f_u(u_*(\mu);\mu)$ and $\cL_0(\mu):=g_u(u_*(\mu);\mu)$ and suppose that $\cL_2(\mu)$ is positive-definite (note that this is not entirely necessary due to the fact that $\cL_2(\mu)$ needn't be symmetric, but is nice to have).

 \subsection{Multiscale Expansion}\label{s:5.1}
 Formally, we can expand the derivatives in \eqref{eq:MGNMasterEqn} as
 \begin{equation}\label{eq:MGNMasterEqnExpanded}
	 u_t=h_u(u;\mu)(u_x,u_x)+h(u;\mu)u_{xx}+f_u(u)u_x+g(u),
 \end{equation}
 where we think of $h_u$ as the bilinear form given by
 \begin{equation}\label{eq:huBil}
	 h_u(u;\mu)(U,V)=\sum_{i,j,k=1}^nh_{ij,u_k}(u;\mu)U_jV_ke_i.
 \end{equation}
 We adopt the convention that $h_{uu}(u;\mu)$ is the trilinear form given by
 \begin{equation}\label{eq:huuTri}
	 h_{uu}(u;\mu)(U,V,W)=\sum_{i,j,k,l=1}^nh_{ij,u_ku_l}(u;\mu)U_jV_kW_le_i.
 \end{equation}
 Writing $u(x,y)=u_*+U(x,t)$ and Taylor expanding $f$, $g$ and $h$ gives then
 \begin{equation}\label{eq:MGNMasterEqnTaylor}
 \begin{split}
	 U_t=& h_u(u_*;\mu)(U_x,U_x)+h_{uu}(u_*;\mu)(U_x,U_x,U)+\cO(|U|^2|U_x|^2)+\\
		 & +h(u_*;\mu)U_{xx}+h_u(u_*;\mu)(U_{xx},U)+\frac{1}{2}h_{uu}(u_*;\mu)(U_{xx},U,U)+\cO(|U_{xx}||U|^3)+\\
		 & +f_u(u_*;\mu)U_x+f_{uu}(u_*;\mu)(U_x,U)+\frac{1}{2}f_{uuu}(u_*;\mu)(U_x,U,U)+\cO(|U_x||U|^3)+\\
		 & +g(u_*;\mu)+g_u(u_*;\mu)U+\frac{1}{2}g_{uu}(u_*;\mu)(U,U)+\frac{1}{6}g_{uuu}(u_*;\mu)(U,U,U)+\cO(|U|^4).
 \end{split}
 \end{equation}
 In this equation, $g_{uu}$ and $g_{uuu}$ are the usual multilinear forms and we write
 \begin{equation}\label{eq:fuBil}
	 f_{uu}(u;\mu)(U,V)=\sum_{i,j,k=1}^nf_{i,u_j,u_k}(u;\mu)U_jV_ke_i
 \end{equation}
 and
 \begin{equation}\label{eq:fuTri}
	 f_{uuu}(u;\mu)(U,V,W)=\sum_{i,j,k,l=1}^nf_{i,u_j,u_k,u_l}(u;\mu)U_jV_kW_le_i.
 \end{equation}

 We take the same Ansatz as in section \eqref{sec:Multi}, as well as the scaling $L(\mu)=L(0)+\e^2\d_\mu L(0)+\cO(\e^4)$. Because the procedure is virtually identical to the one in section \eqref{sec:Multi}, we will only highlight parts of the argument that have significant change. 
 Notice that \eqref{eq:e11} and \eqref{eq:e21} are unchanged, since they were linear.

 Let $U_0=\frac{1}{2}Ae^{i\xi}r+c.c.$. For the analogue of \eqref{eq:e20}, given by
 \begin{equation}\label{eq:loce20}
 	S(0,0)\Psi_0+\widehat{h_u(\d_\xi^2 U_0,U_0)}(0)+\widehat{ h_u(\d_\xi U_0,\d_\xi U_0)}(0)+\widehat{f_{uu}(\d_\xi U_0,U_0)}(0)+\frac{1}{2}\widehat{g_{uu}(U_0,U_0)}(0)=0,
 \end{equation} %Wrote out the equation here -A 
 we need to verify that the nonlinearity present actually coincides with a real function. In the above, we dropped the arguments on $h_u$, $f_{uu}$, $g_{uu}$ for notational clarity; they should all be evaluated at $(u_*,0)$. To verify the reality of the nonlinear terms, we split the nonlinearity into terms containing only $f_{uu}$, $g_{uu}$ and $h_{u}$. %Corrected a small typo here, h_uu forms are trilinear and hence don't appear at order \e^2 -A 
% First, the 
The terms containing $g_{uu}$ were already present in the original form of \eqref{eq:e20} 
and have already been shown to be real. For the terms coming from $f_{uu}$ and $h_u$ we have the following claim.
 \begin{lemma}\label{lem:localbilform}
	 At order $\e^2$, the Fourier coefficient of $\cQ(\d_\xi^j U,\d_\xi^l U)$ at frequency zero, where $\cQ:\RR^n\times\RR^n\to\RR^n$ is a bilinear form, is of the form $|A|^2v$ where $v\in\RR^n$ is a known vector.
 \end{lemma}
 \begin{proof}
 	%Proof of this lemma has been reworked, the direct argument for reality of v has been replaced with a symmetry-type argument -A
	For this order and Fourier mode, it suffices to consider the crude approximation $U(x,t)=\frac{1}{2}\e Ae^{i\xi}r+c.c.$. With this approximation in hand, we see that $\cQ(\d_\xi^j U,\d_\xi^lU)$ is a real-valued function because $\cQ$ and $U$ are real-valued, and hence has mean value in $\RR^n$. But we can also expand $\cQ(\d_\xi^j U,\d_\xi^lU)$ using the bilinearity to get
	\begin{equation}
	\begin{split}
		\cQ(\d_\xi^j U,\d_\xi^lU)=\frac{1}{4}\left[\cQ((ik_*^j Ar),(ik_*)^lAr)e^{2i\xi}+\cQ((ik_*^j Ar),(-ik_*)^l\overline{Ar})\right.\\
		\left.+\cQ((-ik_*^j\overline{Ar}),(ik_*)^lAr)+\cQ((-ik_*^j\overline{Ar}),(-ik_*)^l\overline{Ar})e^{-2i\xi} \right]
	\end{split}
	\end{equation}
	As in \eqref{eq:e20}, the zero Fourier coefficient of $\cQ(\d_\xi^jU,\d_\xi^lU)$ can be expressed as
	\begin{equation}
		\widehat{\cQ(\d_\xi^j U,\d_\xi^lU)}(0)=\frac{1}{4}|A|^2\left[\cQ((ik_*^j)r,(-ik_*)^l\overline{r})+\cQ((-ik_*)^j\overline{r},(ik_*)^lr)\right]
	\end{equation}
	which is easily seen to be of the form $|A|^2v$ where $v$ is a known vector. Note that $v$ is a priori complex, however, we've already established that it is real because $|A|^2v$ is the mean value of a real-valued function and $|A|^2$ is real.
 \end{proof}
 Applying the claim to each bilinear form appearing at order $\e^2e^{0i\xi}$ and arguing as before in 
 \eqref{eq:e20}-\eqref{eq:Psi0}, we can write $\Psi_0=|A|^2v_0$ where $v_0\in\RR^n$. 
 There are no essential changes to \eqref{eq:e21} in this context;
 it is straightforward to check that all terms that give an exponential of $e^{2i\xi}$ at order $\e^2$ are of the form $A^2v$ for some vector $v\in\CC^n$ from \eqref{eq:loce20}.

 The last thing to check is that the nonlinearity in the equation at order $\e^3$ and Fourier mode $e^{i\xi}$ has the form $|A|^2Av_3$ for some $v_3\in\CC^n$. This was already established for all terms with some collection of derivatives of $g$. For the trilinear terms, this is essentially immediate because the only way for the product of three terms in our 
 Ansatz to be $\cO(\e^3)$ it is necessary that each of these terms 
 %NEWEST is 
 be $\cO(\e)$. Once each term is $\cO(\e)$, it is a complex multiple $Ae^{i\xi}$ or $\overline{A}e^{-i\xi}$ and then for them to add to $e^{i\xi}$, it has to be $|A|^2Ae^{i\xi}$. Inspecting \eqref{eq:MGNMasterEqnTaylor}, we see that there are three new types of bilinear terms to handle
 \begin{itemize}
 	\item $f_{uu}(u_*;0)(U_x,U)$
 	\item $h_{u}(u_*;0)(U_x,U_x)$
 	\item $h_{u}(u_*;0)(U_{xx},U)$
 \end{itemize}
 For each type, there are essentially two cases: either the $\cO(\e^2)$ term is one slow derivative of $A$ or $\overline{A}$, or the $\cO(\e^2)$ term is one of $\Psi_0$, $\Psi_1$ or $\Psi_2$. 
 In either case the $\cO(\e)$ term
 %NEWEST
 is $A$ or $\overline{A}$. Note that any number of $\xi$ derivatives can be taken in either case and that $\Psi_0$ must appear as a $U$. The first case is impossible
 %NEWEST impossible, 
 since the frequencies can't add to 1 and in the second case
 %NEWST delete , 
 the only allowable options are essentially $\Psi_0 A$ and $\Psi_2\overline{A}$. But these kinds of terms are of the form $|A|^2Av$ for $v\in\CC^n$. This completes the modifications in the complex Ginzburg-Landau derivation; as everything linear is unchanged.

	\begin{remark}\label{genrmk}
		Heuristically, essentially any translation invariant nonlinearity can be used so long as it's smooth enough and quadratic near 0. One starts as before by Taylor expanding 
		%the nonlinearity as
	%	$\sN(U^\e;k)$ into $\sQ(U^\e,U^\e)+\sC(U^\e,U^\e,U^\e)+\sQ_k(\d_{\Hx}U^\e,U^\e)$ 
		\be\label{nlexpansion}
	\sN(U^\e;k)= \sQ(U^\e,U^\e)+\sC(U^\e,U^\e,U^\e)+\sQ_k(\d_{\Hx}U^\e,U^\e) + h.o.t.,
		\ee
	where $\sQ$, $\sQ_k$ and $\sC$ are translation invariant bilinear and trilinear forms respectively. Now by Proposition \ref{thm:multilin}, translation invariant multilinear operators are given by multilinear multipliers, so the arguments in this section can be modified to allow for general nonlinearities.

		We will provide more details deriving the amplitude equation for general nonlocal nonlinearities in the subsection \eqref{subsec:TwoGammas}, where we also establish that the constant $\g$ in (cGL) is the same as 
a corresponding one coming from Lyapunov-Schmidt reduction. Assuming that the amplitude equation is complex Ginzburg-Landau in this level of generality, it suggests that the key underlying structure that makes the amplitude equation complex Ginzburg-Landau is translation-invariance together with the property that
	that the kernel of the linear operator have complex dimension 1.

	A further remark is that translation-invariance is in some sense playing a dual role. On one hand, it implies that we can take eigenfunctions of $L(k,\mu)$ to be pure exponentials as opposed to essentially arbitrary smooth functions. More importantly, it is also gives some ``compatibility'' between the linear operator $L(k,\mu)$ and the nonlinearity $\sN$ in the sense that the multilinear forms arising from the Taylor expansion of $\sN$ map exponentials to an exponential of a known frequency. This is quite special because even for the simplest nonlinearities there is no reason to expect that the product of eigenfunctions is ever again an eigenfunction.
	\end{remark}

	%TODO: here, check wording
 \subsubsection{Expansion to all orders}\label{s:allorders}
  Here, we will tackle the question of higher order expansions in the multiscale expansion of
  approximate solutions of \eqref{std} with general quasilinear nonlinearity. 
  To do this, we change the notation of our Ansatz to
\begin{equation}\label{eq:higherorderansatz}
U^n=\frac{1}{2}\e Ae^{i\xi}r+c.c.+\sum_{k=2}^n\sum_{\eta=0}^k\frac{1}{2}\e^k(\Psi_{\eta}^ke^{i\eta\xi}+c.c.)+\frac{1}{2}\e^{n+1}\sum_{\eta=0}^{n+1}\Psi_\eta^{n+1}e^{i\eta \xi}+c.c.
\end{equation}
with the hypothesis that $\Psi_0^k$ is real-valued for all $k$. Define $\cA_k:=\ell\Psi_1^k$ to be the amplitude at order $k$. %Not sure if this is the correct name for this -A
Our goal is the following theorem.
\begin{theorem}\label{allordersthm}
For any $n=2,3,4,...$ and any sufficiently smooth $\cA_1$ satisfying the 
complex Ginzburg-Landau equation \eqref{eq:cGL} and amplitudes $\cA_2,...,\cA_{n-1}$ satisfying 
\begin{equation}\label{eq:higherordercgl}
(\cA_k)_{\Ht}=-\frac{1}{2}\tl_{kk}(k_*,0)(\cA_k)_{\Hx\Hx}+\tl_\mu(k_*,0)\cA_k+\g(2|A|^2\cA_k+A^2\bar{\cA}_k)+F_k(A,\cA_2,...,\cA_{k-1})
\end{equation}
on $0\leq \hat t\leq T$,
where the $F_k$ are known smooth functions, then there exists an approximate solution of \eqref{eq:MasterEqn} of the form \eqref{eq:higherorderansatz} and some choices of smooth $\cA_n:=\ell\Psi_1^n$ and $\cA_{n+1}:=\ell\Psi_1^{n+1}$ that is consistent to order $O(\e^{n+1})$ where 
$d_*=-\frac{\Im\tl(k_*,0)}{k_*}$ and $d_*+ \delta=-\Im\d_k\tl(k_*,0)$.
	That is, it has truncation error $O(\e^{n+2})$.
on $0\leq \hat t\leq T$.
\end{theorem}
As before, we cannot claim uniqueness of this approximate solution. We will not pursue the question of whether or not the sequence $U^n$ actually converges.
\begin{proof}
	The first thing we will do is show that the $\Psi_\eta^k$ for $\eta\not=1$ can be constructed in terms of $A,\cA_2,...,\cA_{k-1}$. Note that $\Psi_\eta^k$ first appears as coefficient of $\e^ke^{i\eta\xi}$. Looking at this equation we find
	\begin{equation}\label{eq:PsiEtak}
	[S(k_*\eta,0)+ik_*d_*\eta]\Psi_\eta^k+F_{\eta,k}=0,
	\end{equation}
	where $F_{\eta,k}$ is in principle known and depends only on $\Psi_{\eta'}^{k'}$ with $k'<k$. To be a bit more precise about $F_{\eta,k}$, it has linear terms of the form $\d_k^j\d_\mu^l S(k_*\eta,0)\d_{\Hx}^j\Psi_\eta^{k-j-2l}$. What kinds of nonlinear terms actually appear depends on 
	the original nonlinearity in the system, but because they are all multilinear and the smallest available power of $\e$ is 1; they cannot depend on $\Psi_{\eta'}^k$ for any $\eta'$. So we construct the approximate solution inductively, with base case Section \ref{s:5.1}. From now on, we will focus on $\Psi_1^{n-1}$ and $\Psi_\eta^n$ for $\eta=0,2$ as they are the most important terms in deriving the amplitude equation for $\cA_{n-1}$. 

	%what the original nonlinearity in the system was, but because they are all multilinear and the smallest available power of $\e$ is 1; they cannot depend on $\Psi_{\eta'}^k$ for any $\eta'$. So we construct the approximate solution inductively, with base case Section \ref{s:5.1}. From now on, we will focus on $\Psi_1^{n-1}$ and $\Psi_\eta^n$ for $\eta=0,2$ as they are the most important terms in deriving the amplitude equation for $\cA_{n-1}$. \\
	In addition to existence, we need that $\Psi_0^n$ is real-valued; which we can establish with the following argument. It suffices to show that $F_{0,n}$ is real-valued. For the linear terms of $F_{0,n}$, this follows from the fact that the underlying linear operator of \eqref{eq:MGNMasterEqn} has real coefficients and $\d_\xi$, $\d_{\Hx}$ map real functions to real functions. A bilinear form $\cQ((\d_\xi+\e\d_{\Hx})^JU,(\d_\xi+\e\d_{\Hx})^KU )$ sends real-valued functions to real-valued functions and hence has real Fourier coefficient in mode 0. But then one can consider the cruder approximation $U=U^{n-1}$ which is real-valued by induction. A similar argument will work for any degree of multilinearity.

	Looking at the equation for $\Psi_1^{n+1}$, we find
	\ba \label{eq:Psi^n+1}
	(\Psi_1^{n-1})_{\Ht}&-(d_*+\delta)(\Psi_1^n)_{\Hx}=(S(k_*,0)+id_*k_*)\Psi_1^{n+1}-i\d_kS(k_*,0)(\Psi_1^n)_{\Hx}-
	-\d_k^2S(k_*,0)(\Psi_1^{n-1})_{\Hx\Hx}\\
	&\quad +\d_\mu S(k_*,0)\Psi_1^{n-1}+\frac{1}{4}\sQ(-1,2)(\bar{A}\bar{r},\Psi_2^n)+\frac{1}{2}\sQ(1,0)(Ar,\Psi_0^n)+\\
	&\quad +\frac{1}{4}\sQ(-1,2)(\bar{\Psi}_1^{n-1},\Psi_2^2)+\frac{1}{2}\sQ(0,1)(\Psi_0^2,\Psi_1^{n-1})+
	\frac{1}{16}[\sC(1,1,-1)(Ar,Ar,\overline{\Psi_1^{n-1}})\\
	&\quad +2\sC(1,-1,1)(Ar,\bar{A}\bar{r},\Psi_1^{n-1}) ]+\tilde{F}_{1,n+1}(A,\cA_2,...,\cA_{n-2}),
	\ea
	where we let $\sN$ be the nonlinearity in \eqref{eq:MGNMasterEqn}, $\sQ$ the multiplier of $D_u^2\sN(0)$,
	and $\sC$ the multiplier of $D_u^3\sN(0)$. We've also adopted the shorthand $\sQ(n,m)=\sQ(nk_*,mk_*)$, with the obvious modification for $\sC$. From our inductive argument, we have the expansion
	\begin{equation}
	(I_n-\Pi)\Psi_1^n=i(\cA_{n-1})_{\Hx}N(I_n-\Pi)\d_kS(k_*,0)r+G_n,
	\end{equation}
	where $G_n$ is a known function of $A,\cA_2,...,\cA_{n-2}$ and $N$ is the matrix
	\begin{equation*}
	N=[(I_n-\Pi)(S(k_*,0)+id_*k_*)(I_n-\Pi) ]^{-1}.
	\end{equation*}

	For \eqref{eq:Psi^n+1} to be solvable, it is necessary and sufficient that $\ell$\eqref{eq:Psi^n+1}=0. 
	For the linear terms, this means that
	\begin{equation}
	(\cA_{n-1})_{\Ht}-(d_*+\delta)(\cA_n)_{\Hx}=-i\ell\d_k S(k_*,0)(\cA_n)_{\Hx}r-\frac{1}{2}\tl_{kk}(k_*,0)(\cA_{n-1})_{\Hx\Hx}r+\tl_\mu(k_*,0)\cA_{n-1}+F_{n-1},
	\end{equation}
	where $F_{n-1}$ is a known function of $A,\cA_2,...,\cA_{n-2}$ coming from the $(I_n-\Pi)\Psi_1^n$ terms. Here we've used Lemma \ref{lem:EigenDeriv}to combine the leading order term of $(I_n-\Pi)\Psi_1^n$ with the $\d_k^2S(k_*,0)$ term. As before $(d_*+\delta)(\cA_n)_{\Hx}=i\ell\d_k S(k_*,0)r(\cA_n)_{\Hx}$, so we have the correct linear part for $\cA_{n-1}$.

	For the nonlinear terms, we will adopt the convention that $F_{n-1}$ is a known function that may change from line to line; but only depends on $A,\cA_2,...,\cA_{n-2}$. First, we focus on the quadratic terms of \eqref{eq:Psi^n+1}. There are two types of quadratic terms, those that depend directly on $\Psi_1^{n-1}$, and those that don't. We begin by recalling the formulas for $\Psi_2^2$ and $\Psi_0^2$ from Section \ref{sec:Multi},
	\begin{align*}
		\Psi_0^2&=|A|^2\left(-\frac{1}{4}S(0,0)^{-1}\left[\sQ(1,-1)(r,\overline{r})+\sQ(-1,1)(\overline{r},r)\right] \right) \\
		\Psi_2^2&=-A^2\frac{1}{4}\left(S(2k_*,0)+2ik_*d_*\right)^{-1}\sQ(1,1)(r,r) 
	\end{align*}
	Plugging these into the relevant terms in \eqref{eq:Psi^n+1}, we get for the quadratic terms directly depending on $\Psi_1^{n-1}$
	\begin{equation}
	\begin{split}
		\frac{1}{4}\sQ(-1,2)(\bar{\Psi}_1^{n-1},\Psi_2^2)+\frac{1}{2}\sQ(0,1)(\Psi_0^2,\Psi_1^{n-1})=\\
		=-A^2\bar{\cA}_{n-1}\frac{1}{16}\sQ(2,-1)\left(\left(S(2k_*,0)+2ik_*d_*\right)^{-1}\sQ(1,1)(r,r),r\right)-\\
		-\frac{1}{8}|A|^2\cA_{n-1}\sQ(1,0)(r,S(0,0)^{-1}\sQ(1,-1)(r,\bar{r}))+F
	\end{split}
	\end{equation}
	Turning to the other quadratic terms in \eqref{eq:Psi^n+1}, we expand out the nonlinearity slightly in \eqref{eq:PsiEtak} to get the dependence on $\Psi_1^{n-1}$ in $\Psi_0^n$ and $\Psi_2^n$. For $\Psi_0^n$, we get
	\begin{equation}\label{eq:Psi0n}
		S(0,0)\Psi_0^n+\frac{1}{4}\left(\sQ(1,-1)(\Psi_1^{n-1},\bar{A}\bar{r})+\sQ(1,-1)(Ar,\bar{\Psi}_1^{n-1} \right)+F_{n-1}=0
	\end{equation}
	and for $\Psi_2^n$, we get
	\begin{equation}\label{eq:Psi2n}
		[S(2k_*,0)+2ik_*d_*]\Psi_2^n+\frac{1}{2}\sQ(1,1)(\Psi_1^{n-1},Ar)+F_{n-1}=0.
	\end{equation}

	Note that $\Psi_2^n$ is in some sense twice $\Psi_2^2$, this extra factor of 2 comes from using the symmetry of the form to fix $\Psi_1^{n-1}$ in the first position. So feeding \eqref{eq:Psi0n} and \eqref{eq:Psi2n} into the appropriate terms in \eqref{eq:Psi^n+1}, we get
	\ba
		\frac{1}{4}\sQ(-1,2)(\bar{A}\bar{r},\Psi_2^n)+\frac{1}{2}\sQ(1,0)(Ar,\Psi_0^n)&=
		-\frac{1}{8}|A|^2\cA_{n-1}\sQ(-1,2)(\bar{r},[S(2k_*,0)+2ik_*d_*]^{-1}\sQ(1,1)(r,r))\\
		&\quad -A\frac{1}{8}\sQ(1,0)(r,(\bar{A}\cA_{n-1}+A\bar{\cA}_{n-1})\sQ(1,-1)(r,\bar{r}))+F_{n-1},
		\ea
	where we've noted that $\sQ(1,-1)(r,\bar{r})=\sQ(-1,1)(\bar{r},r)$ by symmetry. To simplify notation, we define the vectors in $\CC^n$
	\begin{equation}
	\begin{split}
		V_0&:=\sQ(1,0)(r,S(0,0)^{-1}\sQ(1,-1)(r,\bar{r})),\\
		V_2&:=\sQ(2,-1)\left(\left(S(2k_*,0)+2ik_*d_*\right)^{-1}\sQ(1,1)(r,r),r\right).
	\end{split}
	\end{equation}
	With these conventions, we can write the quadratic terms of \eqref{eq:Psi^n+1} as
	\begin{equation}
		-\frac{1}{16}A^2\bar{\cA}_{n-1}V_2-\frac{1}{8}|A|^2\cA_{n-1}V_0-\frac{1}{8}|A|^2\cA_{n-1}V_2-\frac{1}{8}(|A|^2\cA_{n-1}+A^2\bar{\cA}_{n-1})V_0+F_{n-1}.
	\end{equation}
	Collecting all of the nonlinear terms of \eqref{eq:Psi^n+1} and applying $\ell$ on the left, we get
	\begin{equation}
		(A^2\bar{\cA}_{n-1}+2|A|^2\cA_{n-1})(-\frac{1}{16}\ell V_2-\frac{1}{8}\ell V_0+\frac{1}{16}\sC(1,1,-1)(r,r,\bar{r}) )+F_{n-1}.
	\end{equation}
	Comparing with the expansion for $\gamma$ given later in Lemma \ref{lem:GeneralGamma} completes the proof, noting that as before, each successive mode is resolved as a bounded function of previous modes and finitely many of their derivatives, hence, by induction, a bounded function of $A$, $\cA_1$,...,$\cA_n$ and their derivatives.
\end{proof}
\begin{remark}
	As before, this can be adapted to the case of a nonlocal system.
\end{remark}

 \subsection{Lyapunov-Schmidt reduction}\label{s:LS2}
 In this section, we consider steady-state solutions to the system \eqref{eq:MGNMasterEqn} in the $\xi=k(x-dt)$ 
 coordinate, i.e., solutions $u$ of
 \begin{equation}\label{eq:MGNLS}
	 0=k^2(h(u;\mu)u_\xi)_\xi+kf(u;\mu)_\xi+g(u;\mu)+dku_\xi,
 \end{equation}
 where $d$ is close to $d_*$ as defined in \eqref{eq:dstar}. Assuming that $u$ is an $H^2$ solution to \eqref{eq:MGNLS}, by Sobolev embedding $u\in C^1$ and hence $h(u)$ is also $C^1$. This allows us to apply the product rule to \eqref{eq:MGNLS} as
 \begin{equation}\label{eq:MGNLS2}
	0=k^2h_u(u;\mu)(u_\xi,u_\xi)+k^2h(u;\mu)u_{\xi\xi}+kf_u(u;\mu)u_\xi+g(u;\mu)+dku_\xi.
 \end{equation}
 Supposing $u=u_*+U$, and adding and subtracting $k^2h(u_*;\mu)u_{\xi\xi}$, $kf_u(u_*;\mu)u_\xi$ and $g_u(u_*;\mu)U$ allows us to rewrite \eqref{eq:MGNLS2} as
 \ba\label{eq:MGNLS3}
	 0=k^2h(u_*;\mu)U_{\xi\xi}+&kf_u(u_*;\mu)U_\xi+dkU_\xi+g_u(u_*;\mu)U+k^2h_u(u;\mu)(U_\xi,U_\xi)+k^2\big(h(u;\mu)
	 \\
	 &\quad -h(u_*;\mu) \big)U_{\xi\xi} +k\big(f_u(u;\mu)-f_u(u_*;\mu) \big)U_\xi+g(u;\mu)-g_u(u_*;\mu)U.
	 \ea

 \begin{remark}
 	$U$ is an $H^2$ solution to \eqref{eq:MGNLS} if and only if it is an $H^2$ solution to \eqref{eq:MGNLS3}.
 \end{remark}

 \begin{lemma}\label{lem:NonlinL2}
 	The nonlinear expression of \eqref{eq:MGNLS3} is in $L^2_{per}(\RR;\RR^n)$. Moreover, if $\cN(U,U_\xi,U_{\xi\xi})$ denotes the nonlinear term in \eqref{eq:MGNLS3}, then there is a constant $R>0$ independent of $U$, and a constant $C>0$ such that $||\cN(U,U_\xi,U_{\xi\xi})||_{L^2_{per}(\RR;\RR^n)}\leq C||U||_{H^2_{per}(\RR;\RR^n)}^2$ for $||U||_{H^2_{per}(\RR;\RR^n)}\leq R$.
 \end{lemma}

 \begin{proof}
 	By Sobolev embedding, we have that $||U||_{L^\infty}+||U_\xi||_{L^\infty}\leq C||U||_{H^2_{per}(\RR;\RR^n)}$ for some $C>0$.

 	Since we have one spatial variable, for all $s>\frac{1}{2}$ and each fixed period $X$ we have that $H^s_{per}([0,X];\RR)$ is an algebra. Hence the first term in the nonlinear expansion $h_u(u;\mu)(U_\xi,U_\xi)\in H^1_{per}(\RR;\RR^n)$ because $h_u(u;\mu)\in C^1$. The best estimate we can have for $||h_u(u;\mu)||_{L^\infty}$ in general is that $||h_u(u;\mu)||_{L^\infty}\leq C$ for $||U||_{L^\infty}$ sufficiently small. So for the first term we have the desired bound
 	\begin{equation}\label{eq:L2Est1}
	 	||h_u(u;\mu)(U_\xi,U_\xi)||_{L^2_{per}(\RR;\RR^n)}\leq C||h_u(u;\mu)||_{L^\infty}||U_\xi||_{H^1_{per}(\RR;\RR^n)}^2\leq C||U||_{H^2_{per}(\RR;\RR^n)}^2.
 	\end{equation}
 	For the second term, we can apply the mean value theorem to conclude
 	\begin{equation}
	 ||h(u;\mu)-h(u_*;\mu)||_{L^\infty}\leq 
		||h_u(u;\mu)||_{L^\infty}||U||_{L^\infty}\leq C||U||_{H^2_{per}(\RR;\RR^n)},
 	\end{equation}
 	so that we may bound
 	\begin{equation}\label{eq:L2Est2}
	 	||\left(h(u;\mu)-h(u_*;\mu) \right)U_{\xi\xi}||_{L^2_{per}(\RR;\RR^n)}\leq||h(u;\mu)-h(u_*;\mu)||_{L^\infty}||U_{\xi\xi}||_{L^2_{per}(\RR;\RR^n)}\leq C||U||_{H^2_{per}(\RR;\RR^n)}^2.
 	\end{equation}

 	An essentially identical estimate gives
 	\begin{equation}\label{eq:L2Est3}
	 	||\left(f_u(u;\mu)-f_u(u_*;\mu) \right)U_\xi||_{L^2(\RR;\RR^n)}\leq C||U||_{H^2_{per}(\RR;\RR^n)}^2.
 	\end{equation}
 	The final term to estimate is $g(u;\mu)-g_u(u;\mu)U$. By Taylor's theorem, we have
 	\begin{equation}\label{eq:L2Est4}
	 	||g(u;\mu)-g_u(u;\mu)U||_{L^2_{per}(\RR;\RR^n)}\leq C||g(u;\mu)-g_u(u;\mu)U||_{L^\infty}\leq C||U||_{L^\infty}^2\leq C||U||_{H^2_{per}(\RR;\RR^n)}^2,
 	\end{equation}
 	provided that $||U||_{L^\infty}$ is small enough. We've also used the observation that all our functions $U$ are $2\pi$ periodic in $\xi$ to get this bound with the first constant independent of $U$.
 	Combining \eqref{eq:L2Est1}, \eqref{eq:L2Est2}, \eqref{eq:L2Est3} and \eqref{eq:L2Est4} with the triangle inequality and the observation that $k=\cO(1)$ we obtain the result.
 \end{proof}

 	\begin{remark}
The proof of this lemma is the only place where we use the quasilinear structure of \eqref{eq:MGNMasterEqn},
		to avoid terms like $U_{\xi\xi}U_{\xi\xi}$ not estimable in $L^2$.
		Though they did not arise in this particular case, terms like $U_\xi U_{\xi\xi}$ can be bounded in the same manner using $||U_\xi U_{\xi\xi}||_{H^2}\leq C||U_\xi||_{\infty}||U_{\xi\xi}||_2$.
 	\end{remark}

 \begin{proposition}\label{prop:FreDiff}
 	The nonlinear operator $\cN$ as defined above is Fr\'echet differentiable.
 \end{proposition}
 \begin{proof}
 	Fix $U\in H^m_{per}(\RR;\RR^n)$, and let $V\in H^m_{per}(\RR;\RR^n)$ with $||V||$ sufficiently small. Recall that $\cN$ is defined by
 	\begin{equation}\label{eq:cNDef}
 	\begin{split}
 	\cN(U,U_\xi,U_{\xi\xi})=k^2h_u(u_*+U;\mu)(U_\xi,U_\xi)+k^2\left(h(u_*+U;\mu)-h(u_*;\mu) \right)U_{\xi\xi}+\\
 	+k\left(f_u(u_*+U;\mu)-f_u(u_*;\mu) \right)U_\xi+g(u_*+U;\mu)-g_u(u_*;\mu)U.
 	\end{split} 	
 	\end{equation}
 	We will show that $\cN$ is Fr\'echet differentiable by working term by term. 
	 %Now in 

	 In what follows, we can always bound $||V||_{L^\infty},||V_\xi||_{L^\infty}$ by $C||V||_{H^2}$ by Sobolev embedding, hence any error term featuring monomials in $||V||_{L^\infty}$,$||V_\xi||_{L^\infty}$ are acceptable error terms. 
	 We start with the first term in \eqref{eq:cNDef},
	 expanding $\cN_1(U):=k^2h_u(u_*+U;\mu)(U_\xi,U_\xi)$ as
 	\ba\label{eq:FDiffN1}
		\cN_1&(U+V)-\cN_1(U)=\\
		& k^2h_u(u_*+U+V;\mu)((U+V)_\xi,(U+V)_\xi) -k^2h_u(u_*+U;\mu)(U_\xi,U_\xi)\\
		& =k^2\big(h_u(u_*+U+V)-h_u(u_*+U) \big)(U_\xi,U_\xi)+k^2\big(h_u(u_*+U+V;\mu)(U_\xi,V_\xi)\\
		&+h_u(u_*+U+V;\mu)(V_\xi,U_\xi) \big)+\cO(||V||_{H^2}^2).
		\ea
	 As we did for $\cN$, we split \eqref{eq:FDiffN1} into terms. For the first term in the above, we can apply Taylor's theorem to $h_u$ to see that
 	\begin{equation}\label{eq:CN11}
	 	(h_u(u_*+U+V)-h_u(u_*+U))(U_\xi,U_\xi)=h_{uu}(u_*+U+V)(U_\xi,U_\xi,V)+\cO(||V||_{L^{\infty}}^2).
 	\end{equation}
	 Next, we consider $h_u(u_*+U+V;\mu)(U_\xi,V_\xi)$, its relative $h_u(u_*+U+V;\mu)(V_\xi,U_\xi)$ being essentially identical. Computing, we have
	\ba \label{eq:CN12}
		h_u(u_*+U+V;\mu)(U_\xi,V_\xi)-h_u(u_*+U;\mu)(U_\xi,V_\xi)&=h_{uu}(u_*+U;\mu)(U_\xi,V_\xi,V)\\
		& =\cO(||V_\xi||_{L^\infty}||V||_{L^\infty}).
		\ea
Combining the results of \eqref{eq:CN11} and \eqref{eq:CN12} we conclude that
		\ba\label{eq:FreDerivCN1}
		D\cN_1(U)(V)&=k^2h_{uu}(u_*+U;\mu)(U_\xi,U_\xi,V)+k^2\big(h_u(u_*+U;\mu)(U_\xi,V_\xi)\\
		&\quad +h_u(u_*+U;\mu)(V_\xi,U_\xi) \big).
		\ea

 	The next term of \eqref{eq:cNDef} that we look at is $\cN_2(U):=k^2\left(h(u_*+U;\mu)-h(u_*;\mu) \right)U_{\xi\xi}$. Expanding it in the same way as we expanded $\cN_1$ gives
 	\begin{equation}\label{eq:CN21}
	 	\cN_2(U+V)-\cN_2(U)=k^2\left(h(u_*+U+V)(U+V)_{\xi\xi}-h(u_*+U;\mu)U_{\xi\xi}\right)-k^2h(u_*;\mu)V_{\xi\xi}.
 	\end{equation}
 	As before, we have that
 	\begin{equation}\label{eq:CN22}
	 	h(u_*+U+V)U_{\xi\xi}-h(u_*+U)U_{\xi\xi}=h_u(u_*+U;\mu)(U_{\xi\xi},V)+\cO(||V||_{L^\infty}^2).
 	\end{equation}
 	For the last remaining interesting part of \eqref{eq:CN21} we have
 	\begin{equation}\label{eq:CN23}
	 	h(u_*+U+V;\mu)V_{\xi\xi}=h(u_*+U;\mu)V_{\xi\xi}+\cO(||V||_{L^\infty}||V_{\xi\xi}||_{L^2}).
 	\end{equation}
 	Combining \eqref{eq:CN22} and \eqref{eq:CN23} gives
 	\begin{equation}\label{eq:FreDerivCN2}
	 	D\cN_2(U)(V)=k^2h_u(u_*+U;\mu)(U_{\xi\xi},V)+k^2h(u_*+U;\mu)V_{\xi\xi}-k^2h(u_*;\mu)V_{\xi\xi}.
 	\end{equation}

 	Let $\cN_3(U):=k\left(f_u(u_*+U;\mu)-f_u(u_*;\mu) \right)U_\xi$. Here we have
 	\begin{equation}\label{eq:CN31}
	 	\cN_3(U+V)-\cN_3(U)=k\left(f_u(u_*+U+V;\mu)(U+V)_\xi-f_u(u_*+U;\mu)U_\xi \right)-kf_u(u_*;\mu)V_\xi.
 	\end{equation}
	 Similarly as before, we have
	 $$
	 \begin{aligned}
	 (f_u(u_*+U+ & V;\mu)-f_u(u_*+U;\mu))U_\xi +f_u(u_*+U+V;\mu)V_\xi=\\
& f_{uu}(u_*+U;\mu)(U_\xi,V)+f_u(u_*+U;\mu)V_{\xi}+\cO(||V_\xi||_{L^\infty}||V||_{L^\infty},||V||_{L^\infty}^2),
	 \end{aligned}
	 $$
giving 
	 \begin{equation}\label{eq:FreDerivCN3}
		 D\cN_3(U)(V)=k\left(f_{uu}(u_*+U;\mu)(U_\xi,V)+f_u(u_*+U;\mu)V_{\xi}-f_u(u_*;\mu)V_\xi \right).
	 \end{equation}

Finally, we treat the last term in \eqref{eq:cNDef}, $\cN_4(U):=g(u_*+U;\mu)-g_u(u_*;\mu)U$.
	 Computing gives
$$
\begin{aligned}
	\cN_4(U+V)-\cN_4(U)&=g(u_*+U+V;\mu)-g(u_*+U;\mu)-g_u(u_*;\mu)V\\
	& =g_u(u_*+U;\mu)V-g_u(u_*;\mu)V+\cO(||V||_{L^\infty}^2),
		 \end{aligned}
		 $$
		 and thus 
	 \begin{equation}\label{eq:FreDerivCN4}
		 D\cN_4(U)(V)=\left(g_u(u_*+U;\mu)-g_u(u_*;\mu) \right)V.
	 \end{equation}

Summing \eqref{eq:FreDerivCN1}, \eqref{eq:FreDerivCN2}, \eqref{eq:FreDerivCN3}, and \eqref{eq:FreDerivCN4} allows 
us to compute $D\cN(U)(V)$, in particular it shows that $\cN$ is Fr\'echet differentiable as desired. One can check that every term in $D\cN(U)$ maps $H^2_{per}(\RR;\RR^n)$ to $L^2_{per}(\RR;\RR^n)$ in a bounded manner. Note the absence 
of terms like $U_{\xi\xi}V_{\xi\xi}$; these would spoil Fr\'echet differentiability in the same way that 
they would spoil boundedness of $\cN$.
 \end{proof}

 The only remaining ingredient in the Lyapunov-Schmidt reduction is to verify $SO(2)$ invariance 
 of \eqref{eq:MGNMasterEqn}. But this a straightforward calculation in position space, where $SO(2)$ 
 acts by $U(\xi)\rightarrow U(\xi-\xi_0)$. Hence the argument in section \eqref{sec:LSRed} goes through with only 
 notational differences, because the only facts about the nonlinearity that were used were 
 $||\cN(U)||_{L^2}\leq C||U||_{H^2}^2$, $SO(2)$ invariance of $\cN(U)$, and Fr\'echet differentiability.

 \begin{remark}
	 In our Lyapunov-Schmidt reduction, 
	 we took care to avoid terms of the form $U_{\xi\xi}U_{\xi\xi}$ because we couldn't bound them in $L^2$.
	 However, the multiscale expansion for complex Ginzburg-Landau can handle these terms without any issues, at least at small orders; the failure of $L^2$ boundedness should become apparent at higher orders only.
	 More generally, one can adapt the arguments in Lemma 
	 \ref{lem:NonlinL2} to show that any nonlinearity of the form $\cN_1(U,\d_\xi U,\d_\xi^2U,...,\d_\xi^{m-1}U)\d_\xi^m U+\cN_2(U,\d_\xi U,...,\d_\xi^{m-1}U)$ is in $L^2_{per}(\RR;\RR^n)$ provided that $\cN_1$ and $\cN_2$ are smooth, $U\in H^m_{per}(\RR;\RR^n)$, $\cN_1(0,0,...,0)=0$, and $\cN_2$ vanishes to quadratic order at $0$.
	 With this discussion in mind, if we insist on $H^{m+1}_{per}(\RR;\RR^n)$ solutions and 
	 view all maps as being from $H^{m+1}_{per}(\RR;\RR^n)\rightarrow H^1_{per}(\RR;\RR^n)$, then there are 
	 no more issues bounding $||(\d_\xi^m U)^2||_{L^2}$ by $||U||_{H^{m+1}}^2$.
	 Thus, we may handle fully general nonlinearities at the price of 
	 further smoothness: perhaps to be expected, as
	 the hard work in the Ginzburg-Landau derivation is 
	 devoted to the linear part, with the nonlinearity barely featured.
 \end{remark}

\subsection{A Tale of Two $\g$'s}\label{subsec:TwoGammas}
In this subsection, we sketch the derivation of the amplitude equation for nonlinearities $\sN:H^s_{per}(\RR;\RR^n)\times\RR\rightarrow L^2_{per}(\RR;\RR^n)$ 
	satisfying the following hypotheses. We then show that the constant $\g$ 
	gotten by formal complex Ginzburg-Landau expansion agrees with the corresponding constant gotten by Taylor 
expansion of the term $n(|\a|^2;k,\mu,d)$ appearing in \eqref{eq:Epseqn2}.

	\begin{hypothesis}\label{hyp:nonlin}
		The nonlinearity satisfies
		\begin{enumerate}
			\item For each $y\in\RR$ let $\t_yf(x):=f(x-y)$. Then for all $u\in H^s_{per}(\RR;\RR^n)$, $\mu\in\RR$, $y\in\RR$ we have $\tau_y\sN(u,\mu)=\sN(\tau_yu,\mu)$. In other words, $\sN$ is translation invariant.\\
			\item For each $X>0$ let $H^s_{per}([0,X];\RR^n)$ denote the subspace of $X$ periodic functions and $k:=\frac{2\pi}{X}$. Then we have isomorphisms $I_k:H^s_{per}([0,2\pi];\RR^n)\to H^s_{per}([0,X];\RR^n)$ given by $I_ku(x)=u(kx)=:u(\xi)$. We assume that $\sN$ is smooth in the sense that the auxiliary map $\sN(u,k,\mu)$ defined by $\sN(u,k,\mu):=I_k^{-1}\sN(I_ku,\mu)$ is smooth as a map from $H^s_{per}([0,2\pi];\RR^n)\times(0,\infty)\times\RR\to L^2_{per}([0,2\pi];\RR^n)$.
			\item $\sN(0,k,\mu)=D_u\sN(0,k,\mu)=0$ for all $k,\mu$.
		\end{enumerate}
	\end{hypothesis}
	The second hypothesis on $\sN$ essentially says every $\sN$ restriction to a subspace of the form $H^s_{per}([0,X];\RR^n)$ is smooth and that the family of restrictions smoothly depends on the period. As an example, if $\sN(u,\mu)=\d_x u$, then $\sN(u,k,\mu)=k\d_\xi u$ by the chain rule. For a (linear) nonlocal example, fix $\varphi\in S(\RR)$ a Schwartz function and consider $\sN(u,\mu):=\varphi*u$ where $f*g=\int f(x-y)g(y)dy$ when the integral is absolutely convergent. Then by the change of variables $z=ky$
	\begin{equation}
		\sN(u,k,\mu)=\int \varphi\left(\frac{x}{k}-y\right)u(ky)dy=\frac{1}{k}\int\varphi\left(\frac{x-z}{k}\right)u(z)dz=\frac{1}{k}(I_k^{-1}\varphi)*u
	\end{equation}
	Morally, one wants to think of the map $H^s_{per}([0,2\pi];\RR^n)\times(0,\infty)\rightarrow H^s_{per}(\RR;\RR^n)$ given by $(u,k)\rightarrow I_ku$ as a homeomorphism with ``inverse'' $u\rightarrow (\overline{u},k)$ where $\frac{1}{k}$ is the minimal period of $u$ and $\overline{u}=I_k^{-1}u$. However, while $(u,k)\rightarrow I_ku$ is a continuous surjection, it dramatically fails to be injective. Equally troubling is that the proposed inverse map is only defined for nonconstant functions and fails to be continuous.
\begin{remark}
In the nonlocal example provided above, we can formally rewrite the nonlinearity as
		\begin{equation}\label{decomp}
			\sN(u,k,\mu)(\xi) 
			\sim \sum_{\eta\in\ZZ}\hat{\varphi}(k\eta)\hat{u}(\eta)e^{i\eta\xi} .
		\end{equation}
	\end{remark}
In the following derivation for $\g$, we are also allowing for general symbols $S(k,\mu)$ satisfying the 
Turing hypotheses. Recalling the expansion
	$$
	\sN(U^\e;k)= 1/2D_u^2\sN(0;k_*)(U^\e,U^\e)+\frac{1}{6}D_u^3\sN(0;k_*)(U^\e,U^\e,U^\e)+\frac{1}{2}\d_k D_u^2\sN(0;k_*)(\d_{\Hx}U^\e,U^\e) + h.o.t.
	$$
of \eqref{nlexpansion}, let $\sQ$ and $\sC$ denote the multipliers for $D_u^2\sN$ and $D_u^3\sN$ respectively,
as guaranteed by Lemma \ref{thm:multilin}.

	\begin{lemma}\label{lem:GeneralGamma}
		%$\g$ is given by
		Informally identifying $\sQ(nk_*,mk_*;0)$ with $\sQ(n,m)$ and similarly for $\sC$, we have
			\ba\label{eq:generalgamma}
		\g&=\ell\big[\sQ(0,1)(-\frac{1}{8}S_0\Re\sQ(1,-1)(r,\bar{r}),r)\\
		&\quad
		+\sQ(2,-1)(-\frac{1}{16}S_2\sQ(1,1)(r,r),\bar{r})+\frac{1}{16}\sC(1,1,-1)(r,r,\bar{r}) \big].
		\ea
		%where informally, we've identified $\sQ(nk_*,mk_*;0)$ with $\sQ(n,m)$ and similarly for $\sC$.
	\end{lemma}
	\begin{proof}
		We start with the following key fact.
		\begin{obs}
			$\d_k^j\sN(0,k,\mu)=\d_k^jD_u\sN(0,k,\mu)=0$ for all $k>0$ and all $j\in\NN$.
		\end{obs}
		This follows from $\sN(0,k)=I_k\sN(I_k^{-1}0)\equiv0$ and $D_u\sN(\overline{u},k)=I_kD_u\sN(I_k^{-1}\overline{u})I_k^{-1}$ which is also identically zero when $\overline{u}=0$.
		We then Taylor expand the nonlinearity, and upon applying the above observation, discover that
		\begin{equation}
		\begin{split}
		\sN(U^\e,k,\mu)=\frac{1}{2}D^2_u\sN(0,k_*,0)(U^\e,U^\e)+\frac{1}{6}D^3_u\sN(0,k_*,0)(U^\e,U^\e,U^\e)+\\+\frac{1}{2}\kappa\d_k D_u^2\sN(0,k_*,0)(U^\e,U^\e)+\frac{1}{2}\mu\d_\mu D_u^2\sN(0,k_*,0)(U^\e,U^\e) +\cO(\e^4).
		\end{split}
		\end{equation}
		Observe that the $\d_\mu$ term is already $\cO(\e^4)$ and thus can be safely ignored.

		Since each form in the above is translation invariant, it follows that each is a multilinear Fourier multiplier operator, which we will denote by
		\begin{equation}
		\begin{split}
		D_u^2\sN(0,k,\mu)(U,V)&=\sum_{\eta_1,\eta_2\in\ZZ}\sQ(k\eta_1,k\eta_2;\mu)(\hat{U}(\eta_1),\hat{V}(\eta_2))e^{i\xi(\eta_1+\eta_2)},\\
		D_u^3\sN(0,k.\mu)(U,V,W)&=\sum_{\eta_1,\eta_2,\eta_3\in\ZZ}\sC(k\eta_1,k\eta_2,k\eta_3;\mu)(\hat{U}(\eta_1),\hat{V}(\eta_2),\hat{W}(\eta_3))e^{i\xi(\eta_1+\eta_2+\eta_3)}.
		\end{split}
		\end{equation}
		Writing $\kappa=\e\o$, we find that 
		\begin{equation}
		\kappa\d_k D_u^2\sN(0,k_*,0)(U^\e,U^\e)=\e \tilde{\sQ}(\d_{\Hx} U^\e, U^\e)
		\end{equation} for some known bilinear form $\tilde{\sQ}$. At $\cO(\e^2)$, 
		the relevant terms are given by
		\begin{equation}\label{eq:GenPsi}
		\begin{split}
		\Psi_2(\Hx,\Ht)=-\frac{1}{4}A(\Hx,\Ht)^2S_2\sQ(k_*,k_*;0)(r,r),\\
		\Psi_0(\Hx,\Ht)=-\frac{1}{8}|A(\Hx,\Ht)|^2S_0\left[\sQ(k_*,-k_*;0)(r,\bar{r})+\sQ(-k_*,k_*;0)(\bar{r},r) \right],
		\end{split}
		\end{equation}
		where we've used the notation $S_\eta=(S(k_*\eta,0)+id_*k_*\eta)^{-1}$ for $\eta\in\ZZ\backslash\{\pm1 \}$.
From this, we conclude that the nonlinearity contributes at $\cO(\e^3)$ and Fourier mode $e^{i\xi}$ the term
		\begin{equation}
		D^2_u\sN(0,k_*,0)(\Psi_0,\frac{1}{2}Ar)+D^2_u\sN(0,k_*,0)(\frac{1}{2}\Psi_2,\frac{1}{2}\bar{Ar})+\frac{1}{2}D^3_u\sN(0,k_*,0)(\frac{1}{2}Ar,\frac{1}{2}Ar,\frac{1}{2}\bar{Ar}).
		\end{equation}
Plugging in \eqref{eq:GenPsi} and applying $\ell$ gives the desired formula.
	\end{proof}

	\begin{remark}
It's perhaps surprising that the nonlinear term in the amplitude equation is always $|A|^2A$ regardless 
of the nature of the original nonlinearity. 
To understand this, we first sketch a reduction to a local system. Note that the 
Ansatz, when plugged into the equations, only 
experiences a finite amount of information about the multipliers $\sQ$, $\sC$, and the symbol $S$. 
Thus, as far as such computations are concerned,
there is no difference between $S$, $\sQ$, and $\sC$ and polynomials that agree at the appropriate points. 
		One has to be a bit careful which polynomials are used in the interpolation to ensure that the resulting multilinear operators are real valued and symmetric, but it is otherwise easy to interpolate. Moreover, it is straightforward to find an interpolating function $P(k,\mu)$ which is polynomial in $k$, smooth in $\mu$ and constant for $|\mu|\geq 1$, and matches the original symbol and the appropriate derivatives at the desired points. 
One may arrange, further, that $P(k,\mu)$ satisfy the Turing hypotheses by subtracting $Cp(k)^NId$ for $C\gg1$ and $N\gg1$ where $p(k):=k^2(k^2-k_*^2)(k^2-4k_*)^2$ and $C,N$ are independent of $\mu$. 

		From now on, therefore, assume that the system is local. For local systems, the relevant multilinear forms look like $\cQ(\d_x^Iu,\d_x^Jv )$ for $I,J\in\NN$ and $\cQ:\RR^n\times\RR^n\to\RR^n$ a fixed bilinear form, or $\cC(\d_x^Iu,\d_x^Jv,\d_x^Kw)$. The only way for a multilinear term to appear at order $\e^3$ is if it is either the product of 3 order $\e$ terms or it is the product of an order $\e^2$ term and an order $\e$ term. In the first case, all possible trilinear terms are $A^3,|A|^2A,|A|^2\bar{A},\bar{A}^3$ which occur at frequencies $3,1,-1,-3$ respectively, 
		these corresponding to the situation that all derivatives fall on exponential factors 
		$e^{\pm i\xi}$.
		So the trilinear contribution to the Ginzburg-Landau equations concerning
		frequencies $\pm 1$ only, is $C|A|^2A$. For bilinear terms, the only terms at order $\e^3$ involving $A_{\Hx}$ are $A_{\Hx}A$ and $A_{\Hx}\bar{A}$ 
		which occur at
		frequencies 2 and 0 respectively, 
		these terms 
		corresponding to the situation that one derivative falls on $A$ and all others fall 
		on the exponential factors $e^{\pm i\xi}$.
		As they do not involve frequencies $\pm 1$, these terms make no contribution 
		to the Ginzburg-Landau equation.
	\end{remark}

	For the moment, we will denote the $\g$ in Lemma \ref{lem:GeneralGamma} by $\g_{\CC GL}$. Note that we have 
	$$
	n(|\a|^2;k,\mu,d)=\g_{LS}|\a|^2+h.o.t.,
	$$
 where $n(|\a|^2;k,\mu,d)\a=\ell\widehat{\sN(\tilde{u}_{\e,\o},k,\mu)}(1)$.
 Our remaining main goal in this section is to 
 establish the following correspondence, rigorously validating the expansion (cGL).

	\begin{theorem}
		$\g_{\CC GL}=\g_{LS}$.
	\end{theorem}
	\begin{proof}
		It will suffice to show that $\widehat{\sN(\tilde{u}_{\e,\o};k,\mu)}(1)=\g_{\CC GL}|\a|^2\a+h.o.t.$. To this end, we Taylor expand $\sN$ as
		\begin{equation}
			\begin{split}
			\sN(\tilde{u}_{\e,\o};k,\mu)=\frac{1}{2}D_u^2\sN(0,k_*,0)(\tilde{u}_{\e,\o},\tilde{u}_{\e,\o})+\frac{1}{6}D_u^3\sN(\tilde{u}_{\e,\o},\tilde{u}_{\e,\o},\tilde{u}_{\e,\o})+\\
			+\frac{1}{2}\o\e\d_k D_u^2\sN(0,k_*,0)(\tilde{u}_{\e,\o},\tilde{u}_{\e,\o})+\frac{1}{2}\mu\d_\mu D_u^2\sN(0,k_*,0)(\tilde{u}_{\e,\o},\tilde{u}_{\e,\o})+h.o.t.
			\end{split}
		\end{equation}
		By construction, $\tilde{u}_{\e,\o}=\e(\frac{1}{2}\a e^{i\xi}r)+c.c.+\cO(\e^2)$, so because each form in the above expansion is a multilinear Fourier multiplier operator we see that in Fourier mode 1 the smallest power of $\e$ is $\e^3$. Moreover, the forms on the latter line don't contribute because $\d_k D_u^2\sN(0,k_*,0)$ would need to contribute an $\e^2e^{i\xi}$ because it's weighted by $\e$, but this can't happen as to get power $\e^2$ one needs to apply $\d_k D_u^2\sN(0,k_*,0)$ to $\frac{1}{2}\a e^{i\xi}r+c.c.$ in both slots. However, because it's a multilinear Fourier multiplier operator, $\d_k D_u^2\sN(0,k_*,0)(\frac{1}{2}\a e^{i\xi}r+c.c.,\frac{1}{2}\a e^{i\xi}r+c.c.)$ is Fourier supported in $\{0,\pm 2\}$. The other multilinear form in the second line is at least order $\e^4$ since $\mu\sim \e^2$. So, we get that
		%\begin{equation} \begin{split}
		\ba\label{eq:nonlocalTaylor}
\widehat{\sN(\tilde{u}_{\e,\o};k,\mu)}(1)&= 
D_u^2\sN(0,k_*,0)(\frac{1}{2}\a e^{i\xi}r,\widehat{\tilde{u}_{\e,\o}}(0))
+D_u^2\sN(0,k_*,0)(\bar{\a}e^{-i\xi}\bar{r},\widehat{\tilde{u}_{\e,\o}}(2)) 
			\\&\quad
			+\frac{1}{16}D_u^3\sN(\a e^{i\xi }r,\a e^{i\xi}r,\bar{\a}e^{-i\xi}\bar{r})
			+h.o.t.
				\ea
				where we've used the symmetry of the forms.
			%\end{split}
		%\end{equation}

		%So 
		Thus, to prove the claim we need to compute $\frac{1}{2}\widehat{\tilde{u}_{\e,\o}}(0)$, $\widehat{\tilde{u}_{\e,\o}}(2)$ and verify that they match $\Psi_0$, $\Psi_2$ to lowest order. In particular, it will suffice to compute $\widehat{\d_\e^2\tilde{u}_{0,\o}}(0), \widehat{\d_\e^2\tilde{u}_{0,\o}}(2)$. By construction, we have
		\begin{equation}\label{eq:tildeueqn}
		L(k,\mu)\tilde{u}_{\e,\o}+d(k,\mu)k\d_\xi\tilde{u}_{\e,\o}+D_u^2\sN(0,k_*,0)(\tilde{u}_{\e,\o},\tilde{u}_{\e,\o})+\cO(\e^3)=0.
		\end{equation}
		Observe that the nonlinearity is $\frac{1}{4}\e^2D_u^2\sN(0,k_*,0)(\a e^{i\xi}r+c.c,\a e^{i\xi}r+c.c)+\cO(\e^3)$, hence to $\cO(\e^2)$ we only have Fourier modes $\{0,\pm1,\pm2 \}$. Plugging in the Taylor series for $\tilde{u}_{\e,\o}$ and Taylor expanding the symbol in \eqref{eq:tildeueqn} shows that
		\begin{equation}
		\frac{1}{2}S(0,0)\widehat{\d_\e^2\tilde{u}_{0,\o}}(0)+\frac{1}{8}|\a|^2\left[\cQ(1,-1)(r,\bar{r})+\cQ(-1,1)(\bar{r},r) \right]=0,
		\end{equation}
		or, equivalently, using the symmetry of $\cQ$,
		\begin{equation}\label{eq:mode0}
		\widehat{\d_\e^2\tilde{u}_{0,\o}}(0)=-\frac{1}{4}|\a^2|S_0\left[\cQ(1,-1)(r,\bar{r})+\cQ(-1,1)(\bar{r},r) \right]=-\frac{1}{2}|\a^2|S_0\Re\cQ(1,-1)(r,\bar{r}),
		\end{equation}
		where $\sQ$ is as in lemma \ref{lem:GeneralGamma}.

		Similarly, we have
		\begin{equation}
		\frac{1}{2}\left[S(2k_*,0)+2ik_*d_* \right]\widehat{\d_\e^2\tilde{u}_{0,\o}}(2)+\frac{1}{8}\a^2\cQ(1,1)(r,r)=0,
		\end{equation}
		or equivalently
		\begin{equation}\label{eq:mode2}
		\widehat{\d_\e^2\tilde{u}_{0,\o}}(2)=-\frac{1}{4}\a^2S_2\cQ(1,1)(r,r).
		\end{equation}
		Plugging \eqref{eq:mode0} and \eqref{eq:mode2} into \eqref{eq:nonlocalTaylor} we get
		\begin{equation}
		\begin{split}
			\widehat{\sN(\tilde{u}_{\e,\o};k,\mu)}(1)&=\e^2|\a|^2\a\left(\sQ(1,0)(\frac{1}{2}r,-\frac{1}{4}S_0\Re\cQ(1,-1)(r,\bar{r}))
			+\sQ(-1,2)(\frac{1}{2}\bar{r},-\frac{1}{8}S_2\cQ(1,1)(r,r))\right.
			\\
			&\left.+\frac{1}{16}\sC(1,1,-1)(r,r,\bar{r})
			+h.o.t.\right).
		\end{split}
		\end{equation}
		
		Comparing with the formula for $\g$ in \eqref{eq:generalgamma}, 
		we have the result. Note the extra $\frac{1}{2}$'s come from $D_u^2\sN(0,k_*)(\frac{1}{2}\a e^{i\xi}r+c.c.,\frac{1}{2}\frac{\d^2\tilde{u}_{0,\o}}{\d\e^2})$.
	\end{proof}
	\begin{corollary}\label{mainprop}
		Theorem \ref{mainLS} holds for $\cN:\RR^n\rightarrow\RR^n$ a smooth function of quadratic order in $u$.
	\end{corollary}
	
	\begin{proof}
		Recall the expansion of $\tilde{u}_{\e,\o}$ from Theorem \ref{thm:LSReduction}.
		\begin{equation}
			\tilde{u}_{\e,\o}=\frac{1}{2}\e\sqrt{\frac{-\frac{1}{2}\o^2\d_k^2\tl(k_*,0)
				-\Re\d_\mu\tl(k_*,0)\tilde{\mu} }{\Re\g_{LS}}}e^{i\xi}r+c.c.+\cO(\e^2)
		\end{equation}
		Applying the preceding theorem $\g_{\CC GL}=\g_{LS}$, we see that the leading order amplitude of $\tilde{u}_{\e,\o}$ matches the one predicted by (cGL) as desired.
	\end{proof}
	\begin{remark}
		With a bit more work using the proof of $\g_{\CC GL}=\g_{LS}$ and making the correct choice for $\ell \Psi_1$ in the Ansatz \eqref{eq:Ansatz}, one can actually show that the second order (in $\e$) terms in the solution from Lyapunov-Schmidt agrees with the corresponding terms in the Ansatz as well.
	\end{remark}

%OLD:
%\bibliographystyle{unsrt}
%\bibliography{GenTurBib}

\end{document}